\documentclass[a4paper,twoside]{article}
\usepackage{a4}
\usepackage{amssymb}
\usepackage{amsmath}
\usepackage[all]{xy}
\usepackage[active]{srcltx}
\usepackage[pagebackref,colorlinks,citecolor=blue,linkcolor=blue,urlcolor=blue]{hyperref}
\usepackage[dvipsnames]{color}
\usepackage{upref}
\allowdisplaybreaks[2] 
\newcount\minutes \newcount\hours
\hours=\time \divide\hours 60 \minutes=\hours
 \multiply\minutes-60
\advance\minutes \time
\newcommand{\klockan}{\the\hours:{\ifnum\minutes<10 0\fi}\the\minutes}
\newcommand{\tid}{\today\ \klockan}
\newcommand{\prtid}{\smash{\raise 10mm \hbox{\LaTeX ed \tid}}}
\renewcommand{\prtid}{}
\makeatletter  \headheight 10pt
\def\sectionmark#1{} 
\def\subsectionmark#1{}
\newcommand{\sectnr}{\ifnum \c@secnumdepth >\z@
                 \thesection.\hskip 1em\relax \fi}
\def\@evenhead{\footnotesize\rm\thepage\hfil\leftmark\hfil\llap{\prtid}}
\def\@oddhead{\footnotesize\rm\rlap{\prtid}\hfil\rightmark\hfil\thepage}
\def\tableofcontents{\section*{Contents} 
 \@starttoc{toc}}
\makeatother
\makeatletter
\def\@biblabel#1{#1.}
\makeatother
\makeatletter
\let\Thebibliography=\thebibliography
\renewcommand{\thebibliography}[1]{\def\@mkboth##1##2{}\Thebibliography{#1}
\addcontentsline{toc}{section}{References}
\frenchspacing 
\setlength{\@topsep}{0pt}
\setlength{\itemsep}{0pt}%
\setlength{\parskip}{0pt plus 2pt}%
} \makeatother
\makeatletter
\def\mdots@{\mathinner.\nonscript\!.%
 \ifx\next,.\else\ifx\next;.\else\ifx\next..\else
 \nonscript\!\mathinner.\fi\fi\fi}
\let\ldots\mdots@
\let\cdots\mdots@
\let\dotso\mdots@
\let\dotsb\mdots@
\let\dotsm\mdots@
\let\dotsc\mdots@
\def\vdots{\vbox{\baselineskip2.8\p@ \lineskiplimit\z@
    \kern6\p@\hbox{.}\hbox{.}\hbox{.}\kern3\p@}}
\def\ddots{\mathinner{\mkern1mu\raise8.6\p@\vbox{\kern7\p@\hbox{.}}%
    \raise5.8\p@\hbox{.}\raise3\p@\hbox{.}\mkern1mu}}
\makeatother
\makeatletter
\let\Enumerate=\enumerate
\renewcommand{\enumerate}{\Enumerate%
\setlength{\itemsep}{0pt}%
\setlength{\parskip}{0pt plus 1pt}%
\renewcommand{\theenumi}{\textup{(\alph{enumi})}}%
\renewcommand{\labelenumi}{\theenumi}%
}
\makeatother
\makeatletter
\let\Itemize=\itemize
\renewcommand{\itemize}{\Itemize%
\setlength{\itemsep}{0pt}%
\setlength{\parskip}{0pt plus 1pt}%
}
\makeatother
\makeatletter
\def\@seccntformat#1{\csname the#1\endcsname.\quad}
\makeatother
\makeatletter
\long\def\@makecaption#1#2{%
  \vskip\abovecaptionskip
  \sbox\@tempboxa{ #1. #2}%
  \ifdim \wd\@tempboxa >\hsize
    #1. #2\par
  \else
    \global \@minipagefalse
    \hb@xt@\hsize{\hfil\box\@tempboxa\hfil}%
  \fi
  \vskip\belowcaptionskip}
\makeatother
\RequirePackage{ifthen}
\newcommand{\authortitle}[3]{\author{#1}\title{#2}%
   \ifthenelse{\equal{#3}{}}{\markboth{#1}{#2}}{\markboth{#1}{#3}}}
\newcommand{\art}[6]{{\sc #1, \rm #2, \it #3 \bf #4 \rm (#5), \mbox{#6}.}}
\newcommand{\auth}[2]{{#1, #2.}}
\newcommand{\artin}[3]{{\sc #1, \rm #2,  in #3.}}

\newcommand{\book}[3]{{\sc #1, \it #2, \rm #3.}}
\newcommand{\AND}{{\rm and }}
\RequirePackage{amsthm}
\newtheoremstyle{descriptive}%
  {\topsep}   
  {\topsep}   
  {\rmfamily} 
  {}          
  {\bfseries} 
  {.}         
  { }         
  {}          
\newtheoremstyle{propositional}%
  {\topsep}   
  {\topsep}   
  {\itshape}  
  {}          
  {\bfseries} 
  {.}         
  { }         
  {}          
\theoremstyle{propositional}
\newtheorem{thm}{Theorem}[section]
\newtheorem{prop}[thm]{Proposition}
\newtheorem{lem}[thm]{Lemma}
\newtheorem{cor}[thm]{Corollary}
\theoremstyle{descriptive}
\newtheorem{deff}[thm]{Definition}
\newtheorem{example}[thm]{Example}
\newtheorem{remark}[thm]{Remark}
\makeatletter
\renewenvironment{proof}[1][\proofname]{\par
  \pushQED{\qed}%
  \normalfont
  \trivlist
  \item[\hskip\labelsep
        \itshape
    #1\@addpunct{.}]\ignorespaces
}{%
  \popQED\endtrivlist\@endpefalse
} \makeatother
\newcommand{\setm}{\setminus}

\renewcommand{\emptyset}{\varnothing}
\newcommand{\Cp}{{C_p}}
\newcommand{\Cinfty}{{C_\infty}}
\newcommand{\CinftyP}{{C_\infty^{\P}}}
\DeclareMathOperator{\dist}{dist}
\DeclareMathOperator{\osc}{osc}
\DeclareMathOperator{\spt}{supp}
\newcommand{\supp}{\spt}
\DeclareMathOperator{\vertical}{vert}
\newcommand*{\coloneq}{\mathrel{:=}}
\newcommand*\upto{\nearrow}
\newcommand{\loc}{_{\rm loc}}
{\catcode`p =12 \catcode`t =12 \gdef\eeaa#1pt{#1}}      
\def\accentadjtext#1{\setbox0\hbox{$#1$}\kern   
                \expandafter\eeaa\the\fontdimen1\textfont1 \ht0 }
\def\accentadjscript#1{\setbox0\hbox{$#1$}\kern 
                \expandafter\eeaa\the\fontdimen1\scriptfont1 \ht0 }
\def\accentadjscriptscript#1{\setbox0\hbox{$#1$}\kern   
                \expandafter\eeaa\the\fontdimen1\scriptscriptfont1 \ht0 }
\def\accentadjtextback#1{\setbox0\hbox{$#1$}\kern       
                -\expandafter\eeaa\the\fontdimen1\textfont1 \ht0 }
\def\accentadjscriptback#1{\setbox0\hbox{$#1$}\kern     
                -\expandafter\eeaa\the\fontdimen1\scriptfont1 \ht0 }
\def\accentadjscriptscriptback#1{\setbox0\hbox{$#1$}\kern 
                -\expandafter\eeaa\the\fontdimen1\scriptscriptfont1 \ht0 }
\def\itoverline#1{{\mathsurround0pt\mathchoice
        {\rlap{$\accentadjtext{\displaystyle #1}
                \accentadjtext{\vrule height1.593pt}
                \overline{\phantom{\displaystyle #1}
                \accentadjtextback{\displaystyle #1}}$}{#1}}
        {\rlap{$\accentadjtext{\textstyle #1}
                \accentadjtext{\vrule height1.593pt}
                \overline{\phantom{\textstyle #1}
                \accentadjtextback{\textstyle #1}}$}{#1}}
        {\rlap{$\accentadjscript{\scriptstyle #1}
                \accentadjscript{\vrule height1.593pt}
                \overline{\phantom{\scriptstyle #1}
                \accentadjscriptback{\scriptstyle #1}}$}{#1}}
        {\rlap{$\accentadjscriptscript{\scriptscriptstyle #1}
                \accentadjscriptscript{\vrule height1.593pt}
                \overline{\phantom{\scriptscriptstyle #1}
                \accentadjscriptscriptback{\scriptscriptstyle #1}}$}{#1}}}}
\newcommand{\de}{\delta}
\newcommand{\eps}{\varepsilon}
\newcommand{\la}{\lambda}
\newcommand{\ga}{\gamma}
\newcommand{\Om}{\Omega}
\renewcommand{\phi}{\varphi}
\newcommand{\p}{{$p\mspace{1mu}$}}
\newcommand{\R}{\mathbf{R}}
\newcommand{\Z}{\mathbf{Z}}
\newcommand{\Q}{\mathbf{Q}}
\newcommand{\Zp}{\mathbf{Z}^+}
\newcommand{\eR}{{\overline{\R}}}
\newcommand{\Np}{N^{1,p}}
\newcommand{\Ninfty}{N^{1,\infty}}
\newcommand{\NinftyP}{N^{1,\infty}(\P)}
\newcommand{\htilde}{\tilde{h}}
\newcommand{\ut}{\tilde{u}}
\newcommand{\gt}{\tilde{g}}
\newcommand{\mus}{\mu_\sharp}
\newcommand{\Ga}{\Gamma}
\newcommand{\Gav}{\Ga_{\vertical}}
\renewcommand{\P}{\mathcal{P}}
\newcommand{\LL}{\mathcal{L}}
\newcommand{\NX}{{N^1X}}
\newcommand{\NXloc}{{N^1\loc X}}
\newcommand*\NtX{{\widetilde{N}^1X}}
\newcommand{\NXP}{{N^1X(\P)}}
\newcommand{\CX}{{C_X}}
\newcommand{\CXP}{{C_X^{\P}}}
\newcommand{\CXOm}{{C_X^\Om}}
\newcommand*\cconc{c_{\scriptscriptstyle{\!\vartriangle}}}
\newcommand{\fa}{0} 
\newcommand{\tr}{1} 
\newcommand{\cpt}{\text{cpt}}
\newcommand{\proper}{\text{proper}}
\newcommand{\clB}{\itoverline{B}}
\numberwithin{equation}{section}
\newcommand{\imp}{\ensuremath{
\mathchoice{\ \Longrightarrow \ }{\Rightarrow}
                {\Rightarrow}{\Rightarrow}}}
\newcommand{\quadeqv}{\ensuremath{
\mathchoice{\quad \Longleftrightarrow \quad }{\Leftrightarrow}
                {\Leftrightarrow}{\Leftrightarrow}} }
\newcommand{\quadimp}{\ensuremath{
\mathchoice{\quad \Longrightarrow \quad }{\Rightarrow}
                {\Rightarrow}{\Rightarrow}}}
\newcommand{\eqvnospace}{\ensuremath{
\mathchoice{\Longleftrightarrow}{\Leftrightarrow}
                {\Leftrightarrow}{\Leftrightarrow}}}
\newcommand{\negimp}{\ensuremath{\textstyle\kern 0em \not\kern 0em\Rightarrow}}
\newcommand{\Longnegrevimp}{{\textstyle\kern 0.5em \not\kern -0.5em\Longleftarrow}}
\newcommand{\Downnegimp}{{\textstyle\kern 0.28em \not\kern -0.28em\Big\Downarrow}} 
\newcounter{saveenumi}
\newenvironment{ack}{\medskip{\it Acknowledgement.}}{}

\begin{document}

\authortitle{Anders Bj\"orn, Jana Bj\"orn and Luk\'a\v{s} Mal\'y}
{Non-quasicontinuous Newtonian functions 
and outer capacities based on Banach function spaces}
{Non-quasicontinuous Newtonian functions 
and outer capacities}
\author{
Anders Bj\"orn \\
\it\small Department of Mathematics, Link\"oping University, SE-581 83 Link\"oping, Sweden\\
\it \small anders.bjorn@liu.se, ORCID\/\textup{:} 0000-0002-9677-8321
\\
\\
Jana Bj\"orn \\
\it\small Department of Mathematics, Link\"oping University, SE-581 83 Link\"oping, Sweden\\
\it \small jana.bjorn@liu.se, ORCID\/\textup{:} 0000-0002-1238-6751
\\
\\
Luk\'a\v{s} Mal\'y \\
\it\small Department of Science and Technology, Link\"oping University,
\\
\it\small SE-601 74 Norrk\"oping, Sweden\/{\rm ;}
\it \small lukas.maly@liu.se, ORCID\/\textup{:} 0000-0003-2083-9180
}

\date{Preliminary version, \today}
\date{}

\maketitle

\noindent{\small
 {\bf Abstract.} We construct various examples of Sobolev-type functions, 
defined via upper gradients in metric spaces, 
that fail to be quasicontinuous or weakly quasicontinuous. 
This is done with 
quasi-Banach function lattices $X$ as the 
function spaces
defining the smoothness of the Sobolev-type functions.
These results are in contrast to the case $X=L^p$ with $1\le p<\infty$,
where all Sobolev-type functions in $\Np$ are known to be quasicontinuous,
provided that the underlying metric space $\P$ is locally complete.
In most of our examples,
$\P$
is a compact subset of $\R^2$
and $X=L^\infty$.
Four
particular examples are the damped topologist's sine curve,
the von Koch snowflake curve, the Cantor ternary set and the Sierpi\'nski
carpet.
We also discuss several related properties, such as whether the Sobolev
capacity is an outer capacity, and how these properties are related.
A fundamental role in these considerations is played by the lack of
the Vitali--Carath\'eodory property.
}

\medskip


\medskip



\section{Introduction}

We assume throughout the paper that $\P = (\P, d, \mu)$
is a metric measure space equipped with a
positive complete Borel measure 
$\mu$ such that $0 < \mu(B) < \infty$ for every ball $B \subset \P$.
Such spaces allow themselves to first order analysis, including 
Sobolev type spaces, geometric analysis and potential theory as in 
Haj\l asz~\cite{Haj-PA}, \cite{Haj03}, 
Heinonen--Koskela~\cite{HeKo98},
Cheeger~\cite{Cheeger},
Shan\-mu\-ga\-lin\-gam~\cite{Sh-rev}, 
Heinonen~\cite{hei-BAMS}
and many subsequent papers,
as well as the books by
Haj\l asz--Koskela~\cite{HaKo},
Heinonen~\cite{heinonen},
Bj\"orn--Bj\"orn~\cite{BBbook}
and Heinonen--Koskela--Shan\-mu\-ga\-lin\-gam--Tyson~\cite{HKSTbook}.

Since the late 1990s it has been an open question whether
all functions in the 
(Sobolev) Newtonian space $\Np(\P)$, $1 \le p <\infty$,
are quasicontinuous, without additional assumptions on $\P$.
It has recently been answered in the affirmative in 
Eriksson-Bique--Poggi-Corradini~\cite[Theorem~1.3]{EB-PC}
when $\P$ is locally complete.

In this paper we shall see
that this is not so in the wider generality
of Newtonian spaces $\NX(\P)$ based on (quasi)-Banach function lattices $X$,
and in particular not for $\Ninfty(\P)$.
A simple example of such a situation is when the metric space $\P\subset[0,1]$ 
is the classical ternary Cantor set, equipped
with the 
$\tfrac{\log 2}{\log 3}$-dimensional
Hausdorff measure.
In this case, the function $\chi_{[0,a]\cap \P} \in\Ninfty(\P)= L^\infty(\P)$
is not quasicontinuous
when $a\in\P$ is not an end point of any of the generating intervals.

The following result gives a more complete picture and
shows the striking difference
between $\Ninfty(\P)$ and $\Np(\P)$ with $1 \le p<\infty$.
See Theorem~\ref{thm-von-Koch} for a more refined version
and Section~\ref{sect-prelim} for the definitions.

\begin{thm} \label{thm-Linfty-intro}
Assume that there are 
$L^\infty$-almost no nonconstant rectifiable
curves in the metric space
$\P$ and that there is a point $x_0 \in \P$
with $\mu(\{x_0\})=0$.
Then $\Ninfty(\P)= L^\infty(\P)$ and
the following are true\/\textup:
\begin{enumerate}
\item
$\chi_{\{x_0\}} \in \Ninfty(\P)$  is weakly quasicontinuous, but not
quasicontinuous.
\item
The capacity $\CinftyP$ is not an outer capacity, not even for sets of zero capacity.
\item
If  there is a closed set $F$ with empty interior and $\mu(F)>0$,
then $\chi_F \in \Ninfty(\P)$  is not
weakly quasicontinuous.
\end{enumerate}

In particular, all of these facts hold if $\P$ is
the von Koch snowflake curve equipped with the 
$\frac{\log4}{\log3}$-dimensional Hausdorff measure $\mu$.
\end{thm}

To appreciate the statement of Theorem~\ref{thm-Linfty-intro}, 
let us turn back to $\Np(\P)$, with  $1 \le p <\infty$.
Unlike the classical Sobolev functions on $\R^n$,
the functions in $\Np(\P)$ 
are  by definition better than arbitrary a.e.-representatives.
Namely, two $\Np$-functions which are equal a.e.\ are equal 
up to a set of zero \p-capacity (that is, quasieverywhere (q.e.)),
because the other a.e.-representatives
do not possess \p-integrable upper gradients.

For metric spaces $\P$ without (or with $L^p$-almost no) 
rectifiable curves
and $1\le p<\infty$, 
the Newtonian space $\Np(\P)$ coincides with $L^p(\P)$ and 
the \p-capacity coincides with the underlying measure on $\P$.
All $u\in\Np(\P)$ are thus quasicontinuous by virtue 
of Luzin's theorem.
We mention already here that in this context, one
crucial difference between $\Np$ and $\Ninfty$ is that $L^p$
has an absolutely continuous norm and
satisfies the
Vitali--Carath\'eodory property (VC, Definition~\ref{def-VC}), 
while $L^\infty$ usually lacks these properties,
see  Section~\ref{sect-ex-Loo}.

Clearly, without rectifiable curves, $\NX(\P)=X(\P)$
for every quasi-Banach function lattice $X$, but this is also the case
for the well-connected  Sierpi\'nski carpet $S_{1/3}$, 
see Example~\ref{ex-Sierpinski-new}.
As a consequence, Theorem~\ref{thm-Linfty-intro} applies also 
to $\P=S_{1/3}$.  

When $\P$ is complete and
has enough rectifiable curves 
(so that a \p-Poincar\'e inequality holds,  and also assuming that
$\mu$ is doubling), it was shown in
Bj\"orn--Bj\"orn--Shan\-mu\-ga\-lin\-gam~\cite[Theorem~1.1 and p.~1198]{BBS5} 
that all functions in $\Np(\P)$, $1 \le p <\infty$, are quasicontinuous.
(This had been conjectured already in the late 1990s
by Shan\-mu\-ga\-lin\-gam~\cite[Corollary~3.1.7]{Sh-PhD}, \cite[Remark~4.4]{Sh-rev},
see also \cite[p.~1198]{BBS5}.)
This result has later been generalized in various ways,  
see~\cite[Theorem~8.2.1]{HKSTbook}, \cite[p.~1190]{BBLeh1}
and
\cite[Theorem~9.1]{BBsemilocal}.   
The assumption of Poincar\'e inequality was 
(in complete doubling spaces and for $p>1$)
removed by Ambrosio et al.\ in~\cite{AmbCD} and~\cite{AmbGS}
(cf.\ \cite[Remark~8.9]{BBsemilocal}), while
Eriksson-Bique--Poggi-Corradini~\cite[Theorem~1.3]{EB-PC} 
recently 
removed most of the earlier assumptions and only require
that $\P$ is locally complete.
The assumption that $\mu(B)<\infty$ for every ball $B$ is however crucial,
as our Example~\ref{ex-int-not-Loo}(b) shows.

In the world of (quasi)-Banach function lattices, the letter $X$ is usually reserved
for the function lattice itself and corresponds in our setting to the $L^p$ 
or $L^\infty$ spaces, 
which control the Sobolev functions and their gradients. 
For the underlying metric space carrying the
Sobolev functions, we therefore use $\P$ rather than $X$, which is 
otherwise the usual notation in analysis on metric spaces
based on $L^p$ as the function space.

Quasicontinuity is closely related to the outer property
of the Sobolev capacity $\CXP$ associated with $\NXP$. 
We also study this property and  how it 
is related to quasicontinuity, as well as some variants
of these properties.
The following theorem summarizes quite well what can happen in the general case.
Here $\eR:=[-\infty,\infty]$.

\begin{thm} \label{thm-qcont-char} 
Consider the  following statements
for general quasi-Banach function lattices $X$ 
and metric spaces $\P$\/\textup{:}
\begin{enumerate}
\renewcommand{\theenumi}{\textup{(\Alph{enumi})}}%
\item \label{b-qcont}
Every $u \in \NX(\P)$ is quasicontinuous.
\item \label{b-outer}
$\CXP$ is an outer capacity.
\item \label{b-qouter}
$\CXP$ is a quasiouter capacity.
\item \label{b-weak-qcont}
Every  weakly quasicontinuous $u:\P \to \eR$ is quasicontinuous.
\item \label{b-outer-zero}
$\CXP$ is an outer capacity for sets of zero capacity.
\item \label{b-repr}
Every $u \in \NXP$ has a quasicontinuous representative $v$, i.e.\
$v=u$ q.e.
\item \label{b-wqcont}
Every $u \in \NXP$ is weakly quasicontinuous.
\setcounter{saveenumi}{\value{enumi}}
\end{enumerate}
Then 
\begin{equation} \label{eq-and}
    (\ref{b-weak-qcont} \text{ and } \ref{b-wqcont}) 
    \quadeqv
  (\ref{b-outer-zero} \text{ and  } \ref{b-repr})
     \quadeqv \ref{b-qcont}.
\end{equation}
and
\begin{equation} 
\label{eq-xy-figure}
\begin{split} 
\xymatrix{
\ref{b-qcont} \ar@{=>}[r]^{\Longnegrevimp}
\ar@{=>}[rdd]
& \ref{b-outer} \ar@{=>}[r]^{\Longnegrevimp}
\ar@{=>}[d]^{\textstyle\kern -0.65em \not}
& \ref{b-qouter} \ar@{=>}[r]^{\Longnegrevimp}
& \ref{b-weak-qcont} \ar@{=>}[r] 
& \ref{b-outer-zero}
\\
& \ref{b-wqcont} 
\ar@{<=}[d]^{\Downnegimp} \\
& \ref{b-repr}\rlap{\textup{.}} 
\ar@{=>}[uurrr]\ar@{=>}[uurrr]|{\textstyle\kern -0.65em\not}
}
\end{split}
\end{equation}
Moreover, there are examples when all of these properties fail.
\end{thm}

\begin{remark} 
In general, we do not know if \ref{b-outer-zero}\imp\ref{b-weak-qcont}
nor if $(\ref{b-outer-zero} \text{ and  } \ref{b-wqcont}) \imp \ref{b-qcont}$.

All our counterexamples $X$ (showing the negated implications
in \eqref{eq-xy-figure}) are normed 
(i.e.\ not only quasinormed) Banach function
lattices and moreover also satisfy 
\ref{Fatou} and~\ref{df:BFL.locL1} in the 
Definition~\ref{deff-BFS} of Banach function spaces,
and most of them
satisfy \ref{df:BFS.finmeasfinnorm}.
It is thus only the implication \ref{b-repr}\imp
\ref{b-outer-zero}
which we do not know if it holds or not 
within the class of Banach function spaces.
But we do know that \ref{b-wqcont}\negimp\ref{b-outer-zero},
by either one of Examples~\ref{ex-2^{-n}},
\ref{ex-int-Loo} and~\ref{ex-top-sine}, 
see Table~\ref{table-ex} in Section~\ref{sect-ex-Loo}.

The underlying metric space $\P$  is a compact subset 
of $\R$ or $\R^2$ in all our counterexamples 
(showing the negated implications
in \eqref{eq-xy-figure},
as well as in \eqref{ex-xy-Linfty} below) except for Example~\ref{ex-6.13} 
(showing that \ref{b-weak-qcont}\negimp\ref{b-qouter}),
in which $\P \subset \R^2$ is unbounded and closed.
\end{remark}

For the special case $X=L^\infty$ we have the following result.
Similarly to Banach function spaces, we do not know if 
\ref{b-repr}\imp\ref{b-outer-zero}
(which in this case is equivalent to
\ref{b-repr}\eqvnospace\ref{b-qcont}).

\begin{thm} \label{thm-qcont-char-Linfty} 
Consider the  statements \ref{b-qcont}--\ref{b-wqcont} 
in Theorem~\ref{thm-qcont-char} for $X=L^\infty$ 
and a general metric space $\P$,
and in addition the following statement\/\textup:
\begin{enumerate}
\renewcommand{\theenumi}{\textup{(\Alph{enumi})}}%
\setcounter{enumi}{\value{saveenumi}}
\item \label{b-H}
$\Cinfty(E)>0$ for every nonempty set $E \subset \P$.
\end{enumerate}
Then \ref{b-outer}--\ref{b-outer-zero} are equivalent
and
\begin{equation}  \label{ex-xy-Linfty}
\begin{split} 
\xymatrix{
    (\ref{b-outer} \text{ and } \ref{b-wqcont}) 
\ar@{<=>}[r]
&\ref{b-qcont} \ar@{=>}[r]^{\Longnegrevimp}
\ar@{=>}[d]
& \ref{b-outer} 
\ar@{<=>}[r]
\ar@{<=}[d]^{\textstyle\text{\rlap{\kern -0.65em $\not$}} \Downnegimp}
& \ref{b-H}
\\
& \ref{b-repr} \ar@{=>}[r]^{\Longnegrevimp} & 
\ref{b-wqcont}\rlap{\textup{.}} 
}
\end{split}
\end{equation}
Moreover, there are examples when all of these properties fail.
\end{thm}

The following result illustrates the importance of the
 Vitali--Carath\'eodory property for some of 
the properties in Theorem~\ref{thm-qcont-char}. 
It also strengthens Proposition~3.5 in~Mal\'y~\cite{MalL4},
which says that \ref{b-outer-zero} holds under the assumptions in
\ref{VC-qcont}.

\begin{thm} \label{thm-VC}
Assume that $X$ has the Vitali--Carath\'eodory property.
\begin{enumerate}
\item \label{VC-outer}
If $\P$ is proper\/
\textup(i.e.\ all closed balls are compact\/\textup),
  then $\CXP$ is an outer capacity, i.e.\ \ref{b-outer} holds.
\item \label{VC-qcont}
If $\P$ is locally compact, then every weakly quasicontinuous
$u:\P \to \eR$ 
is quasicontinuous, i.e.\ \ref{b-weak-qcont} holds.
\end{enumerate}
\end{thm}

Theorem~\ref{thm-Linfty-intro} shows that the
Vitali--Carath\'eodory assumption cannot be dropped 
even if one merely aimed to show in~\ref{VC-outer} that $\CXP$ is 
an outer capacity for sets of zero capacity (i.e.\ \ref{b-outer-zero}).
We do not know if
the assumption that $\P$ is proper can be omitted 
in Theorem~\ref{thm-VC}, nor if the remaining properties
from Theorem~\ref{thm-qcont-char} can fail 
when $X$ has the Vitali--Carath\'eodory property.

However, we have the following result.
That \ref{b-repr} and~\ref{b-wqcont} hold under this assumption
was shown by Mal\'y~\cite[Proposition~4.5]{MalL4}.
Note that $\P$ is not assumed to be locally complete here.

\begin{thm} \label{thm-qcont-char-cont-dense}
Assume that continuous functions are dense in $\NX(\P)$.
Then \ref{b-qcont}--\ref{b-outer-zero} are all equivalent
and \ref{b-repr}--\ref{b-wqcont} hold.
However,
it can happen that \ref{b-qcont}--\ref{b-outer-zero} fail.
\end{thm}

Our 
Example~\ref{ex-int-not-Loo}(a)
showing that \ref{b-qcont}--\ref{b-outer-zero} can fail
even when continuous functions are dense in $\NX(\P)$ 
is a Banach function lattice (satisfying also \ref{Fatou} 
and~\ref{df:BFL.locL1},
but not \ref{df:BFS.finmeasfinnorm}).
We do not know if this can happen 
when  $X$ is a Banach function space.

Theorems~\ref{thm-VC}\ref{VC-outer} and \ref{thm-qcont-char-cont-dense}
give the following result, which appeared already in 
Mal\'y~\cite[Corollary~4.6 and Propositions~3.5, 4.3, 4.5 and~4.8]{MalL4}
(except for \ref{b-outer}, 
which was only deduced under the additional assumption that $X$ is normed).

\begin{cor} \label{cor-qcont}
If $\P$ is locally
compact, $X$ has the Vitali--Carath\'eodory property and 
continuous functions are dense in $\NXP$, then
all functions in $\NXP$ are quasicontinuous.

Hence, it follows from Theorem~\ref{thm-qcont-char}  
that \ref{b-qcont}--\ref{b-wqcont} are all true.
\end{cor}

Therefore, at least one of the three assumptions in Corollary~\ref{cor-qcont}
has to be violated to be able to find a non-quasicontinuous Newtonian
function.

Some of the implications in Theorem~\ref{thm-qcont-char} are
of course trivial.
For the nontrivial ones, Mal\'y~\cite[Propositions~4.3 and~4.8]{MalL4}
has earlier shown that \ref{b-outer}\imp\ref{b-weak-qcont},
and for normed $X$ that \ref{b-qcont}\imp\ref{b-outer}.

Since the $L^p$ norm, with $1 \le p <\infty$, 
 has the Vitali--Carath\'eodory property
(see e.g.\ Mal\'y~\cite[Proposition~2.1]{MalL4}),
Corollary~\ref{cor-qcont}
recovers several known results for the standard Newtonian
$\Np$ spaces, assuming that $\P$ is locally compact and 
such that continuous functions are dense in $\Np(\P)$,
cf.\ 
\cite[Section~5.2]{BBbook},
\cite[Theorem~9.1]{BBsemilocal},
\cite[p.~1190]{BBLeh1},
\cite{BBS5},
\cite{EB-PC}
and
\cite[Remark~7.4.3 and Corollary~8.2.5]{HKSTbook}.

If $u$ has a weakly quasicontinuous representative $v=u$ q.e, then 
$u$ is itself weakly quasicontinuous.
(The proof of this fact is very similar to the proof of
\ref{b-qouter}\imp\ref{b-weak-qcont}.)
Hence, the ``weakly quasicontinuous version'' of 
\ref{b-repr} is equivalent to \ref{b-wqcont}.

The outline of the paper is as follows:
In Section~\ref{sect-prelim}
we introduce and discuss the necessary background 
on Banach function lattices and Newtonian spaces.
Most of the properties mentioned above are defined here.

Section~\ref{sect-1.2-1.4} is devoted to the proofs of 
Theorems~\ref{thm-qcont-char} and~\ref{thm-qcont-char-cont-dense}.
The various counterexamples showing the negated implications
are postponed to later, but a list is given here showing how they apply.
In Section~\ref{sect-1.6} we prove
Theorem~\ref{thm-VC}, while in 
Section~\ref{sect-local} we formulate local versions of 
\ref{b-qcont}, \ref{b-weak-qcont}, \ref{b-repr} and~\ref{b-wqcont},
and show that they are equivalent to their respective global versions.

In Sections~\ref{sect-ex-Loo} and~\ref{sect-ex-non-Loo}
we provide various counterexamples.
The summarizing Table~\ref{table-ex} is also provided. 
Examples based on $X=L^\infty$ are given in Section~\ref{sect-ex-Loo}, while
our non-$L^\infty$  examples are given in Section~\ref{sect-ex-non-Loo}.
Theorems~\ref{thm-Linfty-intro}
and~\ref{thm-qcont-char-Linfty} 
are  also proved in Section~\ref{sect-ex-Loo}.
The last two sections contain Example~\ref{exa:nonqcont-N1p(x)}, based 
on the variable exponent case $X=L^{p(\cdot)}$, 
and Example~\ref{ex-Sierpinski-new} dealing with 
$\Ninfty(\P)$ and $\NXP$ on the Sierpi\'nski carpet $S_{1/3}$.

\begin{ack}
A.~B. resp.\ J.~B. were supported by the Swedish Research Council,
  grants 2020-04011 resp.\ 2022-04048.
We would also like to thank Xiangfeng Yang for suggesting the probabilistic
proof used in Example~\ref{ex-Sierpinski-new}.
\end{ack}

\section{Banach function lattices and Newtonian spaces}
\label{sect-prelim}

We extend the measure $\mu$ as an outer measure defined for all $E \subset \P$.
Since $\mu$ is assumed to satisfy $0 < \mu(B) < \infty$ for every ball $B \subset \P$,
it easily follows that $\mu$ is $\sigma$-finite and 
that  $\P$ is separable and Lindel\"of.

A metric space is \emph{proper} if all 
closed balls
are compact.

A vector space $X = X(\P)$ of 
a.e.-equivalence classes of extended
real-valued measurable functions on $\P$ is a
\emph{quasi-Banach function lattice}
equipped with the quasinorm
$\|\cdot\|_X$ if the following axioms hold:
(Here $\Zp=\{1,2,\dots\}$ and a.e.\ is with respect to $\mu$.)
\begin{enumerate}
  \renewcommand{\theenumi}{(P\arabic{enumi})}
  \setcounter{enumi}{-1}
  \item \label{df:qBFL.initial} $\|\cdot\|_X$ determines the set $X$,
    i.e.\ $X = \{u: \|u\|_X < \infty\}$;
  \item \label{df:qBFL.quasinorm} $\|\cdot\|_X$ is a \emph{quasinorm}, i.e.
  \begin{itemize}
    \item $\|u\|_X = 0$ if and only if $u=0$ a.e.,
    \item $\|au\|_X = |a|\,\|u\|_X$ for every $a\in\R$ and every measurable function $u$,
    \item there is a constant $\cconc \ge 1$, the so-called
      \emph{modulus of concavity}, such that 
      \[
         \|u+v\|_X \le  \cconc(\|u\|_X+\|v\|_X)
      \]
 for all measurable functions $u$ and $v$;
  \end{itemize}
  \item $\|\cdot\|_X$ satisfies the \emph{lattice property}, i.e.\ if
    $|u|\le|v|$ a.e., then $\|u\|_X\le\|v\|_X$;
    \label{df:BFL.latticeprop}
  \renewcommand{\theenumi}{(RF)}
  \item \label{df:qBFL.RF} $\|\cdot\|_X$ satisfies the
    \emph{Riesz--Fischer property}, i.e.\ if $u_n\ge 0$ a.e.\@ for all
    $n\in\Zp$, then (with  $\cconc$ from \ref{df:qBFL.quasinorm})
\[
\biggl\|\sum_{n=1}^\infty u_n \biggr\|_X \le \sum_{n=1}^\infty \cconc^n \|u_n\|_X.
\]
\end{enumerate}
Observe that $X$ contains only functions that are finite a.e., which
follows from \ref{df:qBFL.quasinorm} and \ref{df:BFL.latticeprop}. In
other words, if $\|u\|_X<\infty$, then $|u|<\infty$ a.e.

\medskip

\emph{From now on we will always assume that 
$X$ is a quasi-Banach function lattice.}

\medskip

Throughout the paper, we will also assume that the quasinorm
$\|\cdot\|_X$ is \emph{continuous}, i.e.\ 
\begin{equation} \label{eq-cont-quasinorm}
\lim_{n\to\infty}\|u_n - u\|_X =0 
\quadimp
\lim_{n\to\infty}\|u_n\|_X = \|u\|_X. 
\end{equation}
No generality is lost
by this assumption as the Aoki--Rolewicz theorem (see
Benyamini--Lindenstrauss~\cite[Proposition~H.2]{BenLin} or
Maligranda~\cite[Theorem~1.2]{Mali}) implies that there is always an
equivalent quasinorm that is continuous.

The Riesz--Fischer property is actually equivalent to the completeness
of the quasinormed space $X$, provided that the conditions
\ref{df:qBFL.initial}--\ref{df:BFL.latticeprop} are satisfied and that
the quasinorm is continuous, see \cite[Theorem~1.1]{Mali}.
If $\cconc = 1$, then 
$\|\cdot\|_X$ is a norm and \eqref{eq-cont-quasinorm} is trivial. 
We then drop the prefix \emph{quasi} and 
call $X$ a \emph{Banach function lattice}.

\begin{deff} \label{deff-BFS}
A \textup{(}quasi\/\textup{)}Banach function lattice $X$
is a \emph{\textup{(}quasi\/\textup{)}Banach function space} 
if the following axioms are satisfied as well:
\begin{enumerate}
  \renewcommand{\theenumi}{(P\arabic{enumi})}
  \setcounter{enumi}{2}
\item \label{Fatou}
 $\|\cdot\|_X$ has the \emph{Fatou property}, i.e.\ if
  $0\le u_n \upto u$ a.e., then $\|u_n\|_X\upto\|u\|_X$;
  \item \label{df:BFS.finmeasfinnorm} 
if $E$ is measurable and $\mu(E)<\infty$, then
$\|\chi_E\|_X < \infty$;
 \item   \label{df:BFL.locL1}
if $E$ is measurable and $\mu(E)<\infty$, then
there is $C_E>0$ 
such that 
\[
\int_E |u|\,d\mu \le C_E \|u\|_X 
   \quad \text{for every measurable function $u$}.
\]
  \end{enumerate}
\end{deff}

Note that the Fatou property~\ref{Fatou} implies the 
Riesz--Fischer property~\ref{df:qBFL.RF}.
Axiom \ref{df:BFS.finmeasfinnorm} is equivalent to the condition
that $X$ contains all simple functions with support of finite measure.
Condition~\ref{df:BFL.locL1} implies that $X \subset L^1\loc(\P, \mu)$.

For a detailed treatise on Banach function lattices, 
see Lindenstrauss--Tzafriri~\cite{LinTza}. 
The theory of Banach function spaces is thoroughly developed in 
Bennett--Sharpley~\cite{BenSha} 
and
Pick--Kufner--John--Fu\v{c}\'{i}k~\cite{PiKuJoFu}.

In this paper, we will slightly deviate from this rather usual
definition of (quasi)-Banach function lattices. Namely, we will
consider $X$ to be a vector
space of functions defined
\emph{everywhere} instead of equivalence classes defined a.e.
This is crucial for the notion of ($X$-weak) upper gradient 
  to make sense, since
they are defined by the pointwise inequality~\eqref{eq:ug_def}.
The functional $\|\cdot\|_X$ is then really only a (quasi)seminorm.

A \emph{curve} is a continuous mapping from an interval,
and a \emph{rectifiable} curve is a curve with finite length.
Thus, a rectifiable
curve can be (and we will always assume that it is)
parametrized by arc length $ds$, see e.g.\@
Heinonen~\cite[Section~7.1]{heinonen}.

A statement holds for 
\emph{$X$-a.e.\@}
curve if the family $\Gamma_e$ of
exceptional curves, for which the statement fails, has
\emph{zero $X$-modulus}, i.e.\ if there is a Borel function $\rho \in
X$ such that $\int_\gamma \rho\,ds = \infty$ for every curve $\gamma
\in \Gamma_e$.

\begin{deff}
\label{df:ug}
  Let $u: \P \to \eR$. Then, a Borel function $g: \P \to [0,\infty]$ 
is an \emph{upper gradient} of $u$ if
\begin{equation}
 \label{eq:ug_def}
 |u(\gamma(0)) - u(\gamma(l_\gamma))| \le \int_\gamma g\,ds
\end{equation}
for every nonconstant rectifiable curve $\gamma: [0, l_\gamma]\to\P$. 
To make the formulation   
easier, we  use the convention that
$|(\pm\infty)-(\pm\infty)|=\infty$. 

If $g: \P \to [0,\infty]$ is measurable
and~\eqref{eq:ug_def} holds for 
$X$-a.e.\@ 
nonconstant rectifiable curve
$\gamma: [0, l_\gamma]\to\P$, then $g$ is an \emph{$X$-weak upper
  gradient}.
\end{deff}

Observe that the ($X$-weak) upper gradients are by no means given
uniquely. Indeed, if we have a function $u$ with an ($X$-weak) upper
gradient $g$, then $g+h$ is another ($X$-weak) upper gradient of $u$
whenever $h\ge0$ is a Borel (measurable) function.

\begin{deff}
The \emph{Newtonian space} based on $X$ is the space
\begin{equation}
  \label{eq:def-N1X-norm}
  \NX(\P)=  \NX 
  =
 \Bigl\{u\in X :
  \|u\|_{\NX}\coloneq \| u \|_X + \inf_g \|g\|_X <\infty \Bigr\},
\end{equation}
where the infimum is taken over all upper gradients $g$ of $u$.
\end{deff}

Note that we may equivalently define $\NX$ via $X$-weak upper gradients and take
the infimum over all $X$-weak upper gradients $g$ of $u$ in
\eqref{eq:def-N1X-norm} without changing the value of the Newtonian
quasinorm, see Mal\'y~\cite[Corollary~5.7]{MalL1}.
It was shown by Mal\'y~\cite[Theorem~4.6]{MalL2}
 that the infimum in the definition of $\|u \|_{\NX}$ is attained
 for functions in $\NX$ by a \emph{minimal $X$-weak upper
   gradient} of $u$, which 
is minimal both normwise
 and pointwise (a.e.\@) among all $X$-weak upper gradients  in $X$ of $u$,
 and is uniquely given
up to equality a.e.
The minimal $X$-weak  upper gradient of $u$
will be denoted by $g_u$.

The functional $\|\cdot\|_\NX$ is a quasiseminorm on $\NX$ and a
quasinorm on $\NtX \coloneq \NX/\mathord\sim$, where the equivalence
relation $u\sim v$ is given by $\|u-v\|_\NX = 0$. The modulus of
concavity for $\NX$ (i.e.\ the constant in the triangle inequality
for $\|\cdot\|_{\NX}$) is
equal to $\cconc$, i.e.\ it coincides with the modulus of concavity for $X$. 
Furthermore, the
Newtonian space $\NtX$ is complete and thus a quasi-Banach space,
see Mal\'y~\cite[Theorem~7.1]{MalL1}.

\begin{deff}
\label{df:capacity}
The \emph{\textup{(}Sobolev\/\textup{)} $X$-capacity} of a set
$E\subset \P$ is defined as
\[
    C_X^\P(E) =  \CX(E) = \inf\{ \|u\|_{\NX}: u\ge 1 \mbox{ on }E\}.
\]
We say that a property of points in $\P$ holds
\emph{$\CX$-quasieverywhere \textup{(}$\CX$-q.e.\textup{)}} if the set of exceptional
points has $X$-capacity zero.

The capacity $\CX$ is  a \emph{$c$-quasiouter capacity}, with $c \ge 1$, if
\[
    \inf_{\substack{G\text{ open} \\G \supset E  }} \CX(G)  \le c \CX(E)
    \quad \text{for every set } E \subset \P.
\]
The capacity $\CX$ is a \emph{quasiouter capacity} 
if it is a $c$-quasiouter capacity for some $c$,
and it is an \emph{outer capacity} if it is a $1$-quasiouter capacity.

We also say that $\CX$ is an \emph{outer capacity 
for sets of zero capacity}
if  whenever $\CX(E)=0$ and $\eps>0$,  there is an open set $G \supset E$
with $\CX(G) < \eps$.
\end{deff}

In the theory of $L^p$-spaces and more generally 
quasi-Banach function lattices, it is the sets of
measure zero that are negligible. When working with first-order
analysis, it is instead the Sobolev capacity that provides the needed finer
distinction of small sets.
Namely, 
$\|u\|_\NX = 0$
if and only if $u=0$ $\CX$\mbox{-}q.e., by Mal\'y~\cite[Proposition~6.15]{MalL1}.
Hence, the natural equivalence
classes in $\NtX$ as defined above are given by equality $\CX$-q.e.
Moreover, if $u,v \in \NX$ and $u=v$ a.e., then $u=v$ $\CX$-q.e.,
by \cite[Corollary~6.14]{MalL1}.
Despite the dependence on $X$, we will simply write
\emph{capacity} and \emph{q.e.\@} whenever there is no risk of
confusion of the base function space.
Also the dependence on $\P=(\P,d,\mu)$ is implicit 
e.g.\ in $\NX$ and $\CX$.

It is easy to see that $\CX(\emptyset)=0$ and $\CX(E_1) \le \CX(E_2)$
whenever $E_1 \subset E_2$. 
Moreover, 
$\CX(E) = 0$ implies that $\mu(E) = 0$. 
Mal\'y~\cite[Theorem~3.4]{MalL1} showed that
$\CX$ is countably $\cconc$-quasisubadditive, i.e.
\begin{equation} \label{eq-CX-quasi-subadd}
  \CX \biggl( \bigcup_{j=1}^\infty E_j \biggr) \le \sum_{j=1}^\infty \cconc^j \CX(E_j).
\end{equation}
In particular, $\CX$ is countably subadditive if $\cconc=1$.

Given a measurable subset $E \subset \P$, we can restrict the metric and
measure  to $E$ and define the function lattice $X(E)$ 
by restricting functions in $X$ to $E$ and
by setting
\[
\|u\|_{X(E)} = \|\ut\|_X, \quad \text{where }
  \ut(x) =
	\begin{cases}
	  u(x), & \text{if } x\in E, \\
		0, & \text{otherwise.}
	\end{cases}
\]
The space $\NX(E)$ 
and the capacity $C_X^E:=C_{X(E)}$ are then
defined by considering $E$ as
a metric space in its own right, with 
$X(E)$ defined as above and upper gradients taken within $E$.

\begin{remark} \label{rmk-Lp}
When $X=L^p$, $1 \le p \le \infty$, we abbreviate
the notation and write $\Np=N^1L^p$ and 
$\Cinfty=C_{L^\infty}$.
Note that the capacity denoted $\Cp$, with $1<p <\infty$,
in e.g.\ 
\cite{BBbook} and \cite{HKSTbook},
is comparable (but not equal) to the $p$th power of the capacity 
$C_{L^p}$ considered here.
More precisely,
\[
    2^{1-p} C_{L^p}(E)^p \le \Cp(E) \le C_{L^p}(E)^p.
\]
For most purposes in this paper this does not play any role: 
The notions of zero sets, (weak) quasicontinuity and 
quasiouter capacity  are  the same for both capacities.
However, we do not know if the property of being an 
outer capacity is equivalent for $\Cp$ and $C_{L^p}$ 
when $1<p<\infty$.
For this reason, we keep the notation $C_{L^p}$ for the
capacity defined in Definition~\ref{df:capacity}.
When $p=1$, we have $C_1=C_{L^1}$.

Nevertheless,  (the proof of) Corollary~1.3 in 
Bj\"orn--Bj\"orn--Shan\-mu\-ga\-lin\-gam~\cite{BBS5}
(or \cite[Theorem~5.31]{BBbook}) shows that
\ref{b-qcont} implies that $\Cp$ is an outer capacity, which trivially
implies \ref{b-qouter}.
\end{remark}

\begin{deff}
A function $u: \P \to \eR$ is \emph{weakly quasicontinuous}
if for every $\eps>0$ there is a set $E \subset \P$ with
$\CX(E) < \eps$ such that the restriction $u\vert_{\P \setminus E}$ is
continuous.
If the set
$E$ can be chosen open for every $\eps>0$, then $u$ is
\emph{quasicontinuous}.
\end{deff}

Here and elsewhere, continuous functions are real-valued,
while quasicontinuous and other functions may be $\eR$-valued.

The following property is fundamental in Theorem~\ref{thm-VC}
and Corollary~\ref{cor-qcont}.
At the same time, 
all counterexamples in Sections~\ref{sect-ex-Loo}--\ref{sect-Sierpinski}
exploit the fact that it fails.

\begin{deff}   \label{def-VC}
$X$ has the \emph{Vitali--Carath\'eodory property} (VC) if
\begin{equation}
  \label{eq:VitaliCarath}
	\| u \|_X = \inf \{ \|v\|_X: v\ge |u| \mbox{  and $v$ is lower semicontinuous on $\P$}\}
\end{equation}
for every measurable function $u: \P \to \eR$.
\end{deff}

If $X$ is a rearrangement-invariant space whose 
fundamental function approaches zero at the origin, then it
supports the Vitali--Carath\'eodory property \eqref{eq:VitaliCarath} 
by Mal\'y~\cite[Proposition~2.1]{MalL4}. 
Conversely, if there is a point $z \in \P$ with $\mu(\{z\})=0$, 
then \eqref{eq:VitaliCarath} applied to the function 
$\chi_{\{z\}} \in X$ yields that the 
fundamental function of $X$ approaches zero at the origin.
Such a condition is equivalent to the fact that the
rearrangement-invariant space contains some
essentially unbounded functions.

\section{The proofs of Theorems~\ref{thm-qcont-char}
and~\ref{thm-qcont-char-cont-dense}}
\label{sect-1.2-1.4}

In order to prove 
Theorem~\ref{thm-qcont-char}
we will need the following lemma, which says 
that the quasinorm $\|\cdot\|_{\NX}$ is continuous.
This is obvious if $\|\cdot\|_X$ is a norm
(and thus also $\|\cdot\|_{\NX}$ is a norm).
In the general case it uses
our standing assumption that $\|\cdot\|_X$ is continuous,
see \eqref{eq-cont-quasinorm}.

\begin{lem} \label{lem-cont-NX-norm}
If $\|u_k - u\|_{\NX} \to 0$,
then $\|u_k\|_{\NX} \to \|u\|_{\NX}$
as $k \to \infty$.
\end{lem}

\begin{proof}
It follows directly that
$\|u_k - u\|_{X} \to 0$,
and thus $\|u_k\|_{X} \to \|u\|_{X}$,
by our standing assumption that the quasinorm $\|\cdot\|_X$ is continuous,
see \eqref{eq-cont-quasinorm}.

Next, as $g_u \le g_v + g_{u-v}$ a.e.\ it follows that
$g_{u} - g_v \le g_{u-v}$ a.e.,
and thus also $g_v - g_u \le g_{v-u}=g_{u-v}$ a.e.\ so that
$|g_u -g_v| \le g_{u-v}$ a.e.
Applying this to our sequence we see that
\[
    \|g_{u_k} - g_u\|_{X} \le \|g_{u_k- u}\|_{X}
    \le \|u_k- u\|_{\NX} \to 0,
    \quad \text{as } k \to \infty.
\]
Hence by the continuity of $\|\cdot\|_X$ again,
$\|g_{u_k}\|_{X} \to \|g_u\|_{X}$,
and thus $\|u_k\|_{\NX} \to \|u\|_{\NX}$.
\end{proof}

\begin{proof}[Proof of Theorem~\ref{thm-qcont-char}]
We will first show all the ``positive'' implications.

\ref{b-qcont}$\imp$\ref{b-outer}
Let $E \subset \P$ and $0<\eps<1$.
If $\CX(E)=\infty$, the outer capacity property is trivial, so
we assume that $\CX(E)<\infty$.
Then there is $u \in \NX$ with $u \ge \chi_E$ and
\[
\|u\|_{\NX} < \CX(E)+\eps.
\]
By assumption, $u$ is quasicontinuous and hence
there is a sequence of open sets 
$\{V_k\}_{k=1}^\infty$ with $\CX(V_k) < \frac{1}{k}$ such that
the restrictions $u|_{\P \setm V_k}$ are continuous.
Thus, there are open sets $U_k \subset \P$ such that
\[
          U_k \setm V_k = \{x : u(x) > 1-\eps\} \setm V_k \supset E \setm V_k.
\]
We can also find $v_k \ge \chi_{V_k}$
with $\|v_k\|_{\NX} < \frac{1}{k}$.
Let
\[
          w_k= \frac{u}{1-\eps} + v_k.
\]
Then $w_k \ge 1$ on $(U_k \setm V_k) \cup V_k = U_k \cup V_k$,
an open set containing $E$.
Since 
\[
\biggl\| w_k - \frac{u}{1-\eps}\biggr\|_{\NX} 
= \| v_k\|_{\NX} \to 0 \quad \text{as }k\to \infty,
\]
the continuity of the (quasi)norm $\|\cdot\|_{\NX}$,
given by Lemma~\ref{lem-cont-NX-norm},
implies that 
\[
\|w_k\|_{\NX} \to  \frac{\|u\|_{\NX}}{1-\eps}.
\]
In particular, we can find $k\ge 1$ such that 
\[
  \CX(U_{k} \cup V_{k})  \le \|w_{k}\|_{\NX} < \frac{\|u\|_{\NX}}{1-\eps} + \eps < \frac{\CX(E) + \eps}{1-\eps} + \eps.
\]
Letting $\eps \to 0$ shows that $\CX$ is an outer capacity.

\ref{b-qcont}\imp\ref{b-repr} and
\ref{b-outer}$\imp$\ref{b-qouter} 
These implications are trivial.

\ref{b-qouter}$\imp$\ref{b-weak-qcont}
Let  $u:\P \to \eR$ be weakly quasicontinuous
and $\eps >0$.
Then there is a set $E \subset \P$ such that $u|_{\P \setm E}$ is
continuous, and $\CX(E)<\eps/c$, where $c$ is 
the quasiouter constant of $C_X$.
By \ref{b-qouter}, there is an open set $G \supset E$ such
that $\CX(G) < \eps$. 
As $u|_{\P \setm G}$ is
continuous and $\eps>0$ was arbitrary, 
it follows that $u$ is quasicontinuous.

\ref{b-weak-qcont}$\imp$\ref{b-outer-zero}
Let $E$ be a set with $\CX(E)=0$ and $\eps >0$.
Then $u:=\chi_E$ is weakly quasicontinuous.
By assumption, $u$ is quasicontinuous,
so there is an open set $G'$ such that $\CX(G') < \eps/\cconc$
and $u|_{\P \setm G'}$ is continuous.
For each $x \in E \setm G'$, $u|_{\P \setm G'}$ is continuous at $x$
and $u(x)=1$, so there is an open set $G_x \ni x$ such that
$G_x  \setm G' \subset E$,
i.e.\ $G_x \subset G' \cup E$.
Let
\[
    G=G' \cup \bigcup_{x \in E \setm G'} G_x=G' \cup E,
\]
which is open.
Moreover, by \eqref{eq-CX-quasi-subadd},
\[
    \CX(G) \le \cconc \CX(G') + \cconc^2 \CX(E) < \eps.
\]
Since $\eps>0$ was arbitrary, \ref{b-outer-zero} holds.

\ref{b-repr}\imp\ref{b-wqcont}
Let $u \in \NX$ and $\eps>0$. 
Then there is a quasicontinuous function $v$ such that $v=u$ q.e.,
and thus
an open set $G$ such that $\CX(G)<\eps/\cconc$ and
$v|_{\P \setm G}$ is continuous.
Let $E=\{x:v(x) \ne u(x)\}$.
Then $\CX(E)=0$,
$u|_{\P \setm (G \cup E)}$ is continuous
and 
by \eqref{eq-CX-quasi-subadd},
\[
    \CX(G \cup E) \le \cconc \CX(G) + \cconc^2 \CX(E) < \eps.
\]
Hence  $u$ is weakly quasicontinuous.

(\ref{b-weak-qcont} and  \ref{b-wqcont})\imp\ref{b-qcont} 
This is obvious.

(\ref{b-outer-zero}  and   \ref{b-repr})\imp\ref{b-qcont}.
Let $u \in \NX$ and $\eps>0$.
By  \ref{b-repr},
$u$ has a quasicontinuous representative $v$ such that
$v=u$ q.e.
In particular, there is an open set $G$
such that $\CX(G)<\eps$ and
$v|_{\P \setm G}$ is continuous.
Moreover,
by \ref{b-outer-zero}, there is an open set $G'$
such that $\CX(G')<\eps$ and
$u=v$ in $\P \setm G'$.
It follows from \eqref{eq-CX-quasi-subadd} that $\CX(G \cup G')<2\cconc^2\eps$
and that $u|_{\P \setm (G \cup G')}$ is continuous, i.e.\
$u$ is quasicontinuous.

It remains to discuss the negated implications.

Example~\ref{ex-von-Koch} 
shows that all properties \ref{b-qcont}--\ref{b-wqcont} can fail.

Each of Examples~\ref{ex-2^{-n}},
\ref{ex-int-Loo}
and~\ref{ex-top-sine} shows that
\ref{b-wqcont}\negimp\ref{b-repr}.

Each of Examples~\ref{ex-union-new} 
and~\ref{ex-top-sine+int}
shows that 
\ref{b-outer}\negimp\ref{b-qcont}
and  \ref{b-outer}\negimp\ref{b-wqcont}.

Example~\ref{ex-union-new-b} 
shows that 
\ref{b-qouter}\negimp\ref{b-outer}

Example~\ref{ex-6.13} 
shows that
\ref{b-weak-qcont}\negimp\ref{b-qouter}.

Example~\ref{ex-int-not-Loo}(a) 
shows that
\ref{b-repr}\negimp\ref{b-outer-zero}.
\end{proof}

\begin{proof}[Proof of Theorem~\ref{thm-qcont-char-cont-dense}.]
That \ref{b-repr} and~\ref{b-wqcont} 
hold was shown by Mal\'y~\cite[Proposition~4.5]{MalL4}.
That \ref{b-outer-zero}$\imp$\ref{b-qcont} now follows from
\eqref{eq-and}.
Thus \ref{b-qcont}--\ref{b-outer-zero} are equivalent by 
Theorem~\ref{thm-qcont-char}.

That \ref{b-qcont}--\ref{b-outer-zero} can fail follows 
from Example~\ref{ex-int-not-Loo}(a).
\end{proof}

\section{The proof of Theorem~\ref{thm-VC}}
\label{sect-1.6}

In this and the next section we will use balls.
For 
\[
B=B(x,r):=\{y \in X : d(x,y)<r\}
\] 
we use the notation $\la B=B(x, \la r)$
and $\la \clB=\itoverline{B(x, \la r)}$. 
In metric spaces
it can happen that balls with different centres or
radii denote the same set.
We will, however,
make the convention that a ball $B$ comes with a predetermined
centre and radius.
Unless said otherwise, balls considered in this paper
are assumed to be open.

We will need the following two lemmas when proving Theorem~\ref{thm-VC}.

\begin{lem}
\label{lem:int-rho-lsc}
Assume that $\P$ is proper and let $F \subset \P$ be closed. 
Let $\rho: \P \to [0, \infty]$ be lower semicontinuous and such that 
$\rho \ge c > 0$ in $\P \setminus F$. Let
\[
  u(x) = \min\biggl\{ 1, \inf_\gamma \int_\gamma \rho\,ds \biggr\},
\]
where the infimum is taken over all rectifiable 
curves\/ \textup(including constant ones\/\textup) connecting $x \in \P$ to $F$. 
Then, $u$ is lower semicontinuous, $u = 0$ on $F$, 
and $\rho \chi_{\P \setminus F}$ is an upper gradient of $u$.
\end{lem}

Note that we do not assume that $F$ or $\P \setminus F$ is bounded, 
which distinguishes  this from 
Lemma~3.3 in Bj\"orn--Bj\"orn--Shan\-mu\-ga\-lin\-gam~\cite{BBS5}.

\begin{proof}
That $u = 0$ on $F$ is obvious.
It was proven in~\cite[Lemmas~3.1 and~3.2]{BBS5} (or \cite[Lemma~5.25]{BBbook}) 
that $\rho \chi_{\P \setminus F}$ is an upper gradient of $u$. 
What remains to be shown is that 
$u$ is lower semicontinuous.

Let $B \subset \P$ be an arbitrary open ball. For $k \in \Zp$,
let $F_k = F \cup (\P \setminus k B)$ and define
$u_k(x) = \min\bigl\{ 1, \inf_\gamma \int_\gamma \rho\,ds \bigr\}$,
where the infimum is taken over all 
rectifiable curves (including constant ones) connecting $x\in\P$ to $F_k$. 
Then, each function $u_k$ is lower semicontinuous by \cite[Lemma~3.3]{BBS5} 
(or \cite[Lemma~5.26]{BBbook}) as $\P \setminus F_k$ is bounded. 
The proof will be complete when we show that $u = \sup_{k\ge 1} u_k$,  as the pointwise supremum 
of lower semicontinuous functions is lower semicontinuous.

Observe that we can modify the definition of $u(x)$ by considering only 
rectifiable curves $\gamma: [0, l_\gamma] \to \P$ 
such that $\gamma\bigl([0, l_\gamma)\bigr) \subset \P \setminus F$ in the infimum. 
Moreover, we may restrict the length 
of $\gamma$ in the infimum to be at most $\frac{1}{c}$ 
as the value of the path integral would be greater than or equal to $1$ otherwise. 
For each $x \in \P$, we can then find $k_x \ge 1$ such that 
$B(x, \frac1c) \subset k_x B$ and hence $u(x) = u_{k_x}(x)$. 
By definition, $u_k(x) \le u(x)$ for every $x\in \P$ and $k \in \Zp$. Therefore, $u = \sup_{k\ge 1} u_k$.
\end{proof}

\begin{lem} \label{lem-lattice}
If $g_k$ is an upper gradient of $u_k$, $k=1,\dots,n$,
then
$g=\max_k g_k$ is an upper gradient of $u=\min_k u_k$ and of $v=\max_k u_k$.
\end{lem}

For the reader's convenience,
we provide a simple proof.

\begin{proof}
Let $\ga :  [0,l_\ga] \to \P$ be a nonconstant rectifiable curve.
Then 
\[
   |u(\ga(0))-u(\ga(l_\ga))| \le \max_k |u_k(\ga(0))-u_k(\ga(l_\ga))| 
                           \le \max_k \int_\ga g_k\,ds 
                         \le \int_\ga g\,ds,
\]
and thus $g$ is an upper gradient  of $u$.
The proof for $v$ is similar.
\end{proof}

\begin{proof}[Proof of Theorem~\ref{thm-VC}\ref{VC-outer}]
Let $E \subset \P$. 
If $C_X(E) = \infty$, then there is nothing to prove.
We therefore assume that $\CX(E)<\infty$.

Let $\eps>0$. 
Then, there is $u \in \NX$ with an upper gradient $g \in X$ 
such that $\chi_E \le u \le 1$ and $\|u\|_{X} + \|g\|_X < \CX(E) + \eps$. By the Vitali--Carath\'eodory property, there are lower semicontinuous functions $v\ge u$ and $\rho \ge g$ such that $\|v \|_X + \|\rho\|_X < \CX(E) + \eps$. 
Furthermore, there is $0 <\delta <1$ such that 
\[
  (1 + \delta) (\|v\|_X + \| \rho\|_X) < \CX(E) + \eps.
\]
Define
\[
  w(x) = \bigl((1+\delta) v(x) - \delta\bigr)_{+}, \quad x \in \P,
\]
and set $W_0 = \{ x \in \P: w(x) > 0\}$. 
Then, $\| \chi_{W_0} \|_X \le \frac{2}{\delta} \|v\|_X < \infty$. 
By the continuity of the (quasi)norm of $X$,
see \eqref{eq-cont-quasinorm},
we can find $N \in \Zp$ large enough so that
\[
  \Bigl\| \frac{\chi_{W_0}}{N} \Bigr\|_X < \eps, 
\quad \Bigl\| \rho + \frac{\chi_{W_0}}{N} \Bigr\|_X < \| \rho\|_X + \eps 
\quad \text{and}\quad
  \Bigl\| w + \frac{\chi_{W_0}}{N} \Bigr\|_X < \| w\|_X + \eps.
\]
For each $k=1,2,\dots,N-1$, let $W_k = \{ x \in \P: w(x) > \frac{k}{N}\}$.
These sets are open since $w$ is lower semicontinuous.  
Note also that
\begin{equation}
  \label{eq:est_outside_Wk}
  u(y) \le v(y) \le \frac{\frac{k}{N} + \delta}{1+\delta} 
\quad \text{whenever } y \in \P \setminus W_k,\ k = 0, \dots, N-1.
\end{equation}
Next, define  
\[
  u_k(x) = \min \biggl\{1, \frac{k}{N} + \inf_\ga \int_\ga \biggl((1+\delta)\rho 
+ \frac{1}{N}\biggr)\,ds \biggr\}, \quad x \in \P,\ k = 0, \dots, N-1,
\]
where the infimum is taken over all 
rectifiable  curves (including constant ones)
that connect $x$ to the closed set $\P \setminus W_k$. 
Then, $u_k \equiv \frac{k}{N}$ in $\P \setminus W_k$, $k = 0, \dots, N-1$.
Moreover, if $\gamma$ is a rectifiable 
curve with endpoints $x \in E$ and $y \in \P \setminus W_k$, 
then by \eqref{eq:est_outside_Wk},
\begin{align*}
  \int_\gamma \biggl((1+\delta)\rho + \frac{1}{N}\biggr)\,ds 
  & \ge (1+\delta) \int_\gamma \rho\,ds \ge (1+\delta) (u(x) - u(y)) \\
  & \ge (1+\delta) \biggl( 1 - \frac{\frac{k}{N} + \delta}{1+\delta} \biggr) 
= 1 - \frac{k}{N}.
\end{align*}
Therefore, $u_k(x) = 1$ for every $x \in E$ and every $k=0,1,\dots, N-1$. 
All the functions $u_k$ are lower semicontinuous and 
$\gt:= (1+\delta)\rho + \frac{\chi_{W_0}}{N}$ 
is an upper gradient of each $u_k$ by Lemma~\ref{lem:int-rho-lsc}.
Let
\[
  \ut = \min_{0 \le k \le N-1}  u_k
 \quad \text{and} \quad
G = \biggl\{x \in \P: \ut(x) > \frac{1}{1+\eps}\biggr\}.
\]
Then, $\ut$ is lower semicontinuous, 
$\ut(x) = 1$ for all $x\in E$, and 
$\gt$  
is an upper gradient of $\ut$, by Lemma~\ref{lem-lattice}.
Moreover,  $G$ is open, $E \subset G$ and
(with $W_N = \emptyset$)
\begin{align*}
  C_X(G) & \le (1+\eps) \bigl(\|\ut\|_X + \|\gt\|_X \bigr) \\
  & \le (1 + \eps) \biggl( \biggl\| \sum_{k=1}^{N} \frac{k}{N} 
    \chi_{W_{k-1} \setminus W_{k}} \biggr\|_X 
    + (1+\delta) \biggl\| \rho + \frac{\chi_{W_0}}{N} \biggr\|_X \biggr) \\
  & \le (1 + \eps) \biggl( \biggl\| w + \frac{\chi_{W_0}}{N}\biggr\|_{X} 
+ (1+\delta) \biggl\| \rho + \frac{\chi_{W_0}}{N}\biggr\|_{X}\biggr) \\
  & \le (1 + \eps) \bigl( \| w\|_X + \eps + (1+\delta) (\| \rho\|_X + \eps)\bigr) \\
  & \le (1 + \eps) \bigl( (1+\delta)(\| v\|_X +\| \rho\|_X) + 3\eps)\bigr) \\
  & < (1 + \eps) (C_X(E) + 4\eps).
\end{align*}
Letting $\eps \to 0$ shows that $\CX$ is an outer capacity.
\end{proof}

\begin{proof}[Proof of Theorem~\ref{thm-VC}\ref{VC-qcont}]
Using separability and local compactness, we can cover $\P$ by balls
$\{B_j\}_{j=1}^\infty$ with radii $r_j \le 1$ so that the balls
$3 \itoverline{B}_j$ are compact, and thus proper.

Let $u$ be a weakly quasicontinuous function and let $\eps>0$ be arbitrary. 
For each $j \in \Zp$, there is a set $E_j \subset \P$ such that
$C_X(E_j)<(2\cconc)^{-j-1}\eps r_j$ and 
$u|_{\P\setm E_j}$ is continuous.
The Vitali--Carath\'eodory property is inherited by 
the function lattice $X(3\itoverline{B}_j)$, and thus
Theorem~\ref{thm-VC}\ref{VC-outer}
can be used
with 
$X(3\itoverline{B}_j)$.
Upon noting that 
\[
C^{3\itoverline{B}_j}_X(E_j\cap B_j) \le C_X(E_j) < (2\cconc)^{-j-1}\eps r_j,
\]
one can find open sets
$U_j\subset B_j$ such that $E_j\cap B_j \subset U_j$ and
\[
C^{3B_j}_X(U_j) \le
C^{3\itoverline{B}_j}_X(U_j) 
< (2\cconc)^{-j-1}\eps r_j.
\]
Hence, there exist $v_j\in\NX(3B_j)$ such that
\[
\chi_{U_j} \le v_j \le 1 \text{ in } 3B_j
\quad \text{and} \quad
\|v_j\|_{X(3B_j)} + \|g_{v_j}\|_{X(3B_j)} < (2\cconc)^{-j-1}\eps r_j,
\]
where $g_{v_j}$ is the $X$-weak upper gradient of $v_j$ within $3B_j$.
We extend $v_j$ and $g_{v_j}$ by $0$ outside $3B_j$.
Now, let $0\le\eta_j\le1$ be a $(1/r_j)$-Lipschitz cut-off function
vanishing outside
$2B_j$ and such that $\eta_j=1$ in $B_j$.
Putting $u_j=v_j\eta_j$, we have $u_j\in\NX$,
\[
  \chi_{U_j}\le u_j\le \chi_{2 B_j} 
	\quad \text{and}\quad
	g_{u_j} \le g_{v_j} \eta_j + v_j \frac{\chi_{2B_j \setm
            \itoverline{B}_j}}{r_j} \le g_{v_j} + \frac{v_j}{r_j},
\]
by the product rule for $X$-weak upper
gradients, see Mal\'y~\cite[Theorem~A.1]{MalL4}.
Thus,
\begin{align*}
C_X(U_j) &\le \|u_j\|_{X} + \|g_{u_j}\|_{X}
    \le \|v_j\|_{X} + \biggl\|g_{v_j}+\frac{v_j}{r_j}\biggr\|_{X} \\
    &\le \biggl( 1+\frac{\cconc}{r_j} \biggr) \|v_j\|_{X} + \cconc\|g_{v_j}\| _{X}
    < \frac{2\cconc}{r_j} (2\cconc)^{-j-1}\eps r_j = (2\cconc)^{-j}\eps.
\end{align*}
Let $U = \bigcup_{j=1}^\infty U_j \supset \bigcup_{j=1}^\infty E_j\cap B_j$.
It follows from \eqref{eq-CX-quasi-subadd} that
\[
  C_X(U)= C_X\biggl(\bigcup_{j=1}^\infty U_j\biggr)
 \le \sum_{j=1}^\infty \cconc^j C_X(U_j)
< \sum_{j=1}^\infty  2^{-j}\eps =  \eps.
\]
It remains to show that $u|_{\P\setm U}$ is continuous.
Let $x\in \P\setm U$ be arbitrary and find $B_j$ such that $x\in B_j$.
Then $x\in B_j\setm E_j$ and as $u|_{\P\setm E_j}$ is continuous,
the result follows.
\end{proof}

\section{Equivalent local properties}
\label{sect-local}

If $\Om$ is open,
we say  that $f \in \NXloc(\Om)$ if
for every $x \in \Om$ there is a ball $B_x\ni x$ such that
$f \in \NX(B_x)$.

\begin{thm} \label{thm-qcont-char-local}
Consider the  following statements\/\textup{:}
\begin{enumerate}
\renewcommand{\theenumi}{\textup{(\Alph{enumi}$'$)}}%
\item \label{b-Omega}
If $\Om \subset \P$ is open and $u \in \NXloc(\Om)$,
then $u$ is quasicontinuous.
\addtocounter{enumi}{2}
\item \label{b-weak-qcont-Om}
If $\Om \subset \P$ is open and $u:\Om \to \eR$ is weakly quasicontinuous, then
$u$ is quasicontinuous.
\addtocounter{enumi}{1}
\item \label{b-repr-local}
If $\Om \subset \P$ is open then
every $u \in \NXloc(\Om)$ has a quasicontinuous representative $v$, i.e.\
$v=u$ q.e.\ in $\Om$.
\item \label{b-w-Omega}
If $\Om \subset \P$ is open and $u \in \NXloc(\Om)$,
then $u$ is weakly quasicontinuous.
\end{enumerate}
Then
\[
\ref{b-Omega}\eqvnospace\ref{b-qcont}, \quad
\ref{b-weak-qcont-Om}\eqvnospace\ref{b-weak-qcont}, \quad
\ref{b-repr-local}\eqvnospace\ref{b-repr}
\quad \text{and} \quad
\ref{b-w-Omega}\eqvnospace\ref{b-wqcont}.
\]
\end{thm}

Here we consider (weak) quasicontinuity with respect to $\CX$.
Since $\CXOm \le \CX$ the statements 
corresponding to 
\ref{b-Omega}, \ref{b-repr-local} and \ref{b-w-Omega}
with
respect to $\CXOm$ are also equivalent.
For \ref{b-weak-qcont-Om} this is not so clear.

To prove this
we will need the following lemma, which
says that quasicontinuity is a local property.

\begin{lem} \label{lem-qcont-subset}
If $u$ is\/ \textup{(}weakly\/\textup{)}
quasicontinuous in the open set\/ $\Om_j$ for every $j \in \Zp$,
then $u$ is\/ \textup{(}weakly\/\textup{)}
quasicontinuous in\/ $\Om=\bigcup_{j=1}^\infty \Om_j$.
\end{lem}

\begin{proof}
Assume that $u$ is quasicontinuous in each $\Om_j$ and let
$\eps >0$.
For each $j \in \Zp$,
there is an open 
set $G_j$ with $\Cp(G_j)<\eps/(2\cconc)^j$
such that $u|_{\Om_j \setm G_j}$ is continuous.
Let $G=\bigcup_{j=1}^\infty G_j$.
Then $G$ is open 
 and $\Cp(G)<\eps$,
by \eqref{eq-CX-quasi-subadd}.
Since $u|_{\Om \setm G}$ is continuous,
it follows that $u$ is quasicontinuous in $\Om$.
The proof for weak quasicontinuity is the same without the
requirement that $G_j$ are open.
\end{proof}

\begin{proof}[Proof of Theorem~\ref{thm-qcont-char-local}]
The implications 
\ref{b-Omega}\imp\ref{b-qcont},
\ref{b-weak-qcont-Om}\imp\ref{b-weak-qcont},
\ref{b-repr-local}\imp\ref{b-repr}
and 
\ref{b-w-Omega}\imp\ref{b-wqcont}
are trivial
(just apply the statement to $\Om=\P$).

\ref{b-qcont}\imp\ref{b-Omega}
Let $\Om \subset \P$ be open and  $u \in \NXloc(\Om)$.
Since $\P$ is Lindel\"of, we can cover $\Om$ by balls
$B_j=B(x_j,r_j)$ so that $3B_j \subset \Om=\bigcup_{j=1}^\infty B_j$
and $u \in \NX(3B_j)$, $j \in \Zp$.
Let $0\le\eta_j\le1$ be a $(1/r_j)$-Lipschitz cut-off function
vanishing outside
$2B_j$ and such that $\eta_j=1$ in $B_j$.
Putting $u_j=u\eta_j$ (extended by $0$ outside $\Om$),
we can use the product rule for $X$-weak upper
gradients, see Mal\'y~\cite[Theorem~A.1]{MalL4},
to see that
\begin{equation} \label{eq-product}
	g_{u_j} \le g_{u} \eta_j + |u| g_{\eta_j}
              \le g_{u} + \frac{|u|}{r_j}
              \quad \text{a.e.\ in } 2B_j.
\end{equation}
Since $g_{u_j} \equiv 0$ a.e.\ 
in $3B_j \setm 2B_j$, by \cite[Corollary~A.4]{MalL4},
the zero extension of $g_{u_j}$ to $\P \setm 3B_j$ 
is  an $X$-weak upper gradient of $u$ in $\P$.
Hence,
we conclude
that $u_j \in \NX$, and  is thus quasicontinuous.
In particular, $u=u_j$ is quasicontinuous in $B_j$.
It thus follows from Lemma~\ref{lem-qcont-subset} that
$u$ is quasicontinuous in $\Om$.

\ref{b-wqcont}\imp\ref{b-w-Omega}
The proof of this implication is a simple modification
of the proof of \ref{b-qcont}\imp\ref{b-Omega}.

\ref{b-weak-qcont}\imp\ref{b-weak-qcont-Om}
Let $\Om \subset \P$ be open and  
$u:\Om \to \eR$ be weakly quasicontinuous.
Since $\P$ is Lindel\"of, we can cover $\Om$ by balls
$B_j=B(x_j,r_j)$ so that $3B_j \subset \Om=\bigcup_{j=1}^\infty B_j$.
Let $0\le\eta_j\le1$ be a $(1/r_j)$-Lipschitz cut-off function
vanishing outside
$2B_j$ and such that $\eta_j=1$ in $B_j$.
Then $u_j:=u\eta_j$ (extended by $0$ outside $\Om$),
is weakly quasicontinuous in $3B_j$.
As $u_j$  vanishes outside $2B_j$,
it follows that $u_j$ is weakly quasicontinuous in $\P$,
and hence, by assumption,  quasicontinuous in $\P$.
In particular, $u=u_j$ is quasicontinuous in $B_j$.
It thus follows from Lemma~\ref{lem-qcont-subset} that
$u$ is quasicontinuous in $\Om$.

\ref{b-repr}\imp\ref{b-repr-local}
Let $\Om \subset \P$ be open and  $u \in \NXloc(\Om)$.
Extend $u$ by $0$ outside $\Om$.
Since $\P$ is Lindel\"of, we can cover $\Om$ by balls
$B_j=B(x_j,r_j)$ so that $3B_j \subset \Om=\bigcup_{j=1}^\infty B_j$
and $u \in \NX(3B_j)$, $j \in \Zp$.
Let $0\le\eta_j\le1$ be a Lipschitz cut-off function
vanishing outside
$2B_j$ and such that $\eta_j=1$ in $B_j$.
Let recursively, 
\[
   \phi_1=\eta_1 
  \quad \text{and} \quad
\phi_j= \biggl( \eta_j-  \sum_{i=1}^{j-1} \phi_i\biggr)_+, \  j=2,3,\dots.
\]
It follows that
$0 \le \phi_j \le 1$ and that $\phi_j$ is Lipschitz and vanishes 
in $\bigcup_{i=1}^{j-1} B_i$ and
outside  $\bigcup_{i=1}^j 2B_i$.
Moreover,
\[
    \sum_{j=1}^k \phi_j = \max_{1 \le j \le k} \eta_j
\quad \text{and so } 
\sum_{j=1}^\infty \phi_j =1 \text{ in }\Om,
\]
i.e.\ $\{\phi_j\}_{j=1}^\infty$ form a partition of unity in $\Om$.
In a similar way as above, cf.\ \eqref{eq-product}, we see that
$u\phi_j \in \NX$. 
By assumption, there is a quasicontinuous representative $\ut_j$ 
such that $\ut_j=u\phi_j$ q.e.\ in $\P$.
We may assume that $\ut_j \equiv 0$ outside 
$\supp \phi_j:=\overline{\{x:\phi_j(x) \ne 0\}}$.
Then $\ut:=\sum_{j=1}^\infty \ut_j=u$ q.e.\ in $\Om$.
Moreover, since $\phi_j=0$ in $B_k$ whenever $j>k$, we have
\[
    \ut=\sum_{j=1}^k \ut_j
    \quad \text{in } B_k.
\]
Because finite sums of quasicontinuous functions are quasicontinuous 
(due to \eqref{eq-CX-quasi-subadd} as in the proof of Lemma~\ref{lem-qcont-subset}), 
$\ut$ is quasicontinuous in each $B_k$ and thus in $\Om$, 
by  Lemma~\ref{lem-qcont-subset}.
\end{proof}

\section{The case \texorpdfstring{$X=L^\infty$}{X=Loo}}
\label{sect-ex-Loo}

\begin{prop} \label{prop-Linfty-gen}
Assume that $X=L^\infty$.
Then the following are true for any metric space $\P$ 
satisfying the standing assumptions\/\textup{:}
\begin{enumerate}
\renewcommand{\theenumi}{\textup{(\arabic{enumi})}}%
\item \label{r1}
$\Cinfty(G)=1$ for every nonempty open set $G$.
\item \label{r2}
Every quasicontinuous function is  continuous.
\item \label{r3} 
For each $E \subset \P$ either
$    \Cinfty(E)=0$ or $\Cinfty(E)=1$. 
\end{enumerate}
\end{prop}

\begin{proof}
\ref{r1}
As $G$ has positive measure we see that
$\|u \|_{\NinftyP} \ge  \|u \|_{L^\infty} \ge 1$ for any $u \ge \chi_G$,
and so $\Cinfty(G) \ge 1$.
The function $u \equiv 1$ then shows that $\Cinfty(G) = 1$.

\ref{r2}
Let $u$ be quasicontinuous. 
Then there is an open set $G$ such that $\Cinfty(G) < \tfrac12$
and 
$u|_{\P \setm G}$ is continuous.
By \ref{r1}, the open set $G$ must be empty, and thus $u$ is continuous.

\ref{r3}
The function $u \equiv 1$ shows that
$\Cinfty(E) \le 1$.
Assume therefore that $\Cinfty(E)<1$.
Then there is $u \in \NinftyP$ with an upper gradient $g \in L^\infty(\P)$ 
such that $\chi_E \le u \le 1$ and $\|u\|_{L^\infty(\P)} + \|g\|_{L^\infty(\P)} < 1$.
Let 
\[
      v=\frac{(u-a)_+}{1-a}
\quad \text{and} \quad 
      g'=\frac{g}{1-a},
\qquad \text{where } a=\|u\|_{L^\infty(\P)}<1. 
\]
Then $g' \in L^\infty(\P)$ is an upper gradient of $v \in L^\infty(\P)$,
and thus $v \in \NinftyP$.
Moreover $v=1$ on $E$.
We also see that $v=0$ a.e.\ and thus q.e.,
by Mal\'y~\cite[Corollary~6.14]{MalL1} 
(see alternatively the discussion after Definition~\ref{df:capacity}).
Therefore
\[
      \CX(E) \le \|v\|_{\NinftyP}=0.
\qedhere
\]
\end{proof}

\begin{proof}[Proof of Theorem~\ref{thm-qcont-char-Linfty}]
\ref{b-H}\imp\ref{b-outer}
By 
Proposition~\ref{prop-Linfty-gen}\ref{r3},
every nonempty set $E$ has $\Cinfty(E)=1$,
from which it immediately follows that $\Cinfty$ is an outer capacity.

$\neg$\ref{b-H}\imp$\neg$\ref{b-outer-zero}
Let $E$ be a nonempty set with $\Cinfty(E)=0$.
Since $\Cinfty(G)=1$ for every open set $G \supset E$, 
by Proposition~\ref{prop-Linfty-gen}\ref{r1},
we see that $\Cinfty$ is not an outer capacity for zero sets.

Together with Theorem~\ref{thm-qcont-char}
this shows all the ``positive'' implications.

It remains to discuss the negated implications.

Example~\ref{ex-von-Koch} 
shows that all properties \ref{b-qcont}--\ref{b-H} can fail.

Each of Examples~\ref{ex-2^{-n}},
\ref{ex-int-Loo}
and~\ref{ex-top-sine}
shows that 
\ref{b-wqcont}\negimp\ref{b-outer}  and 
\ref{b-wqcont}\negimp\ref{b-repr}.

Each of Examples~\ref{ex-union-new}
and~\ref{ex-top-sine+int}
shows that 
\ref{b-outer}\negimp\ref{b-qcont}
and  \ref{b-outer}\negimp\ref{b-wqcont}.
\end{proof}

The following result 
gives a rather detailed description of the case
when $X=L^\infty$ and $\P$ contains no nonconstant rectifiable curves.
It contains Theorem~\ref{thm-Linfty-intro} as a special case.

\begin{thm} \label{thm-von-Koch}
Assume that $X=L^\infty$ and that there are 
$L^\infty$-almost no nonconstant rectifiable curves
in $\P$.
Then $\NinftyP=L^\infty(\P)$ and
the following are true\/\textup{:}
\begin{enumerate}
\renewcommand{\theenumi}{\textup{(\Roman{enumi})}}%
\item \label{q0}
The capacity is given by
\begin{equation} \label{eq-Cinfty-0/1}
    \Cinfty(E)=\begin{cases}
      0, & \text{if } \mu(E)=0, \\
      1, & \text{if } \mu(E)>0.
      \end{cases}
\end{equation}
\item \label{q1}
If $A$ is nonempty and $\mu(A)=0$, then 
$\chi_A \in \Ninfty(\P)$  is weakly quasicontinuous, but not
quasicontinuous, i.e.\ \ref{b-weak-qcont} fails.
\item \label{q2}
If  there is a closed set $F$ with empty interior and $\mu(F)>0$,
then $\chi_F \in \Ninfty(\P)$  is not
weakly quasicontinuous, i.e.\ \ref{b-wqcont} fails.
\end{enumerate}

Moreover, 
\ref{b-outer}--\ref{b-outer-zero} and \ref{b-H}, which
are equivalent by Theorem~\ref{thm-qcont-char-Linfty},
are also equivalent to 
the following statement\/\textup{:}
\begin{enumerate}
\renewcommand{\theenumi}{\textup{(\Alph{enumi}$'$)}}%
\setcounter{enumi}{\value{saveenumi}}
\item \label{t1}
$\mu(\{x\})>0$ for every $x \in \P$.
\end{enumerate}
\end{thm}

There are two specific examples of what may happen.

\begin{example} \label{ex-von-Koch}
Let $\P$ be 
the von Koch snowflake curve equipped with the 
$\frac{\log4}{\log3}$-dimensional Hausdorff measure $\mu$,
and let $X=L^\infty$.
Then there are no (nonconstant) rectifiable curves and there is a
set $F$ as in \ref{q2} in Theorem~\ref{thm-von-Koch}.
Consequently the following hold:
\begin{itemize}
\item
The function $\chi_F \in \Ninfty(\P)$ is not weakly quasicontinuous,
i.e.\ \ref{b-wqcont} fails.
\item
$C_\infty$ is not an outer capacity for sets of zero capacity, 
i.e.\ \ref{b-outer-zero} fails.
\end{itemize}
Hence \ref{b-qcont}--\ref{b-H}
fail.
\end{example}

\begin{example} \label{ex-2^{-n}}
Let $X=L^\infty$,
\[ 
   \P=\{0,2^{-n}:n \in \Zp\}
\quad \text{and} \quad
\mu=\sum_{n=1}^\infty 2^{-n} \de_{2^{-n}},
\]
where $\de_x$ is the Dirac measure at $x$.
The last part of Theorem~\ref{thm-von-Koch} 
shows that \ref{b-outer-zero} fails.
Let $u \in \Ninfty(\P)$.
Then $u$ is continuous except 
possibly at $0$.
As $\Cinfty(\{0\})=0$, by Theorem~\ref{thm-von-Koch}\ref{q0},
we see that $u$ is weakly quasicontinuous, i.e.\ \ref{b-wqcont} holds.
In particular, \ref{b-wqcont}\negimp\ref{b-outer-zero}.

The function $\chi_E$, where $E=\{2^{-n}:n=2,4,\dots\}$,
shows that continuous functions are not dense in $\Ninfty(\P)$
and that not all functions in $\NinftyP$ have quasicontinuous representatives,
i.e.\ \ref{b-repr} fails.
In particular, \ref{b-wqcont}\negimp\ref{b-repr}.
\end{example}

\begin{proof}[Proof of Theorem~\ref{thm-von-Koch}]
Since there are 
$L^\infty$-almost no
nonconstant rectifiable curves in $\P$,
all functions have 0 as an $L^\infty$-weak upper gradient and 
\[
\NinftyP=L^\infty(\P).
\]

\ref{q0}
That \eqref{eq-Cinfty-0/1} holds is easily verified in a similar way as the proof
of Proposition~\ref{prop-Linfty-gen}\ref{r1}.

\ref{q1}
The function $\chi_{A} \in L ^\infty(\P)=\Ninfty(\P)$ 
is weakly quasicontinuous (as $\Cinfty(A)=0$).
On the other hand, $A$ is not open (since $\mu(B)>0$ for every ball $B$)
and thus $\chi_{A}$ is not continuous, and 
thus 
not quasicontinuous either, by Proposition~\ref{prop-Linfty-gen}\ref{r2}.

\ref{q2}
The function  $u=\chi_F \in L ^\infty(\P)=\Ninfty(P)$. 
Assume that $u$ is weakly quasicontinuous.
Then there is a set $A$ such that $\Cinfty(A)< \tfrac12$ 
and $u|_{\P \setm A}$ is continuous.
By \eqref{eq-Cinfty-0/1}, $\mu(A)=0$ and thus there is $x \in F \setm A$.
Since $u(x)=1$ and $u|_{\P \setm A}$ is continuous,
there is $r>0$ such that $u >\tfrac12$ in $B(x,r) \setm A$.
In particular 
\[
B(x,r) \subset \itoverline{B(x,r) \setm A} \subset \itoverline{F} = F,
\]
which contradicts the choice of $F$ with empty interior.
Hence $u$ is not weakly quasicontinuous.

Finally, \ref{b-H}\eqvnospace\ref{t1}
by \eqref{eq-Cinfty-0/1}.
\end{proof}

\begin{proof}[Proof of Theorem~\ref{thm-Linfty-intro}]
This follows directly from Theorem~\ref{thm-von-Koch}
and Example~\ref{ex-von-Koch}.
\end{proof}

The rest of the paper is devoted to 
more examples.
For the reader's convenience, 
the properties holding
in each example are summarized in Table~\ref{table-ex}.

\begin{table}[t]
\begin{center}
\begin{tabular}
{|l|@{\,}c@{\,}|@{\,}c@{\,}|@{\,}c@{\,}|@{\,}c@{\,}|@{\,}c@{\,}|@{\,}c@{\,}|@{\,}c@{\,}|@{\,}c@{\,}|@{\,}c@{\,}|@{\,}c@{\,}|@{\,}c@{\,}|@{\,}c@{\,}|}
\hline
Example & \ref{b-qcont} & \ref{b-outer} & \ref{b-qouter} & \ref{b-weak-qcont} & 
\ref{b-outer-zero} & \ref{b-repr} & \ref{b-wqcont} & \ref{df:BFS.finmeasfinnorm}
  & (AC) & 
$C(\P)$ & $X=$ & $\P$ \\ 
&&&&&&&&&& dense& $L^\infty$ & \\ \hline
\ref{ex-von-Koch}  & \fa & \fa & \fa & \fa & \fa& \fa & \fa &   \tr & \fa  
& \fa & \tr & \cpt \\ \hline  
\ref{ex-2^{-n}} & \fa & \fa & \fa & \fa & \fa & \fa & \tr &  \tr & \fa  
& \fa & \tr & \cpt\\ \hline
\ref{ex-union-new} & \fa &  \tr & \tr & \tr & \tr & \fa & \fa  & \tr & \fa  
& \fa & \tr & \cpt\\ \hline
\ref{ex-int-Loo} & \fa & \fa & \fa & \fa & \fa & \fa & \tr & \tr & \fa  
& \fa & \tr & \cpt \\ \hline
\ref{ex-top-sine} & \fa & \fa & \fa & \fa & \fa & \fa & \tr & \tr & \fa  
& \fa & \tr & \cpt\\ \hline  
\ref{ex-top-sine+int} & \fa & \tr & \tr & \tr & \tr & \fa & \fa & \tr & \fa  
& \fa & \tr & \cpt\\ \hline  
\ref{ex-union-new-b} & \fa &  \fa & \tr & \tr & \tr & \fa & \fa  & \tr & \fa  
& \fa & \fa & \cpt \\ \hline
\ref{ex-int-not-Loo}(a) 
& \fa & \fa & \fa & \fa & \fa & \tr & \tr & \fa & \tr
& \tr & \fa & \cpt\\ \hline
\ref{ex-int-not-Loo}(b)* 
& \fa & \fa & \fa & \fa & \fa & \tr & \tr & \tr & \tr 
& \tr & \fa & \cpt\\ \hline
\ref{ex-int-not-Loo}(c) 
& \fa & \fa & \fa & \fa & \fa & \fa & \tr & \tr & \fa  
& \fa & \fa & \cpt\\ \hline
\ref{ex-6.13}
& \fa &  \fa & \fa & \tr & \tr & \fa & \fa  & \tr & \fa  
& \fa & \fa & \proper
\\ \hline
\ref{exa:nonqcont-N1p(x)} & \fa & \fa & \fa & \fa & \fa & \fa & \tr & \tr & \fa
& \fa & \fa & \cpt\\ \hline  
\ref{ex-Sierpinski-new} & \fa & \fa & \fa & \fa & \fa & \fa & \fa  & \tr & \fa 
& \fa  & \tr & \cpt\\ \hline   
\end{tabular}
\end{center}
\caption{This table summarizes the properties
each example has.
Here $0= \text{False}$ and $1=\text{True}$.
\newline
* Note that Example~\ref{ex-int-not-Loo}(b) does \emph{not} satisfy
our standing assumption that balls have finite measure.
}
\label{table-ex}
\end{table}

All our examples  are normed Banach function
lattices (i.e.\ not only quasinormed) and satisfy
\ref{Fatou} and~\ref{df:BFL.locL1} 
but fail the Vitali--Carath\'eodory property (VC)
(so these properties are not listed in the table).
The properties \ref{df:BFS.finmeasfinnorm} and
absolute continuity (AC)  are for $X$,
while the density of continuous functions is for $\NX$. 
In the last column we mark if $\P$ is compact
(cpt) or proper.
Note that if \ref{b-repr} or~\ref{b-wqcont} fails then continuous functions
cannot be dense in $\NX$, by Theorem~\ref{thm-qcont-char-cont-dense}.

In this section we concentrate on examples
based on $L^\infty$.
Note that $L^\infty$ is a Banach function space, but the norm
is not absolutely continuous (AC) (unless $\P$ is finite)
and does not  satisfy the Vitali--Carath\'eodory property (VC)
(if there is $x_0 \in \P$ with $\mu(\{x_0\})=0$).
(Note that such a point
$x_0$ cannot be isolated as we have assumed that balls have positive 
measure.)

\begin{deff}  \label{def-AC}
The (quasi)norm $\| \cdot \|_X$ is \emph{absolutely continuous} 
(AC) if 
\[
\| u \chi_{E_n} \|_X \to 0 \quad \text{as } n\to\infty,
\]
whenever $u \in X$ and
$\{E_n\}_{n=1}^\infty$ is a decreasing sequence of measurable sets
with $\mu(\bigcap_{n=1}^\infty E_n) = 0$.
\end{deff}

Mal\'y~\cite[Proposition~2.1]{MalL4} has shown that
(AC)+\ref{df:BFS.finmeasfinnorm}\imp(VC), under our standing assumption
that balls have finite measure.

\begin{example} \label{ex-union-new}
Consider  the compact space
\begin{equation*} 
  \P = \bigcup_{n=0}^\infty I_n,
  \quad \text{where} \quad
  I_n=\begin{cases}
     [-1,0], & \text{if } n=0, \\
     [2^{-1-2n}, 2^{-2n}], & \text{if } n\ge 1,
    \end{cases}
\end{equation*}
endowed with the Euclidean distance and $1$-dimensional Lebesgue measure.
Let
$X=L^\infty$.
Note that the Vitali--Carath\'eodory property (VC) 
fails for $\chi_{\{0\}}$.

Each function in $\NinftyP$ is Lipschitz continuous within each interval $I_n$,
and thus continuous on $\P$, except possibly at $0$ (from the right).
It follows that if $v\in\NinftyP$ satisfies $v \ge 1$ on a nonempty set $E\subset\P$, 
then 
\[
\|v\|_{\NinftyP} \ge \|v\|_{L^\infty} \ge 1
\quad \text{and thus $\Cinfty(E) \ge 1$.}
\]
Since clearly $\Cinfty(E) \le \|1 \|_{\NinftyP} =1$, we see that $\Cinfty(E)=1$.

Thus $\Cinfty$ is an outer capacity, i.e.\ \ref{b-outer} holds, and
any weakly quasicontinuous function must be continuous.
But $\chi_{I_0} \in \NinftyP$ is not continuous, and thus not
weakly quasicontinuous either, i.e.\ \ref{b-wqcont} fails.
Moreover, $\chi_{I_0}$ does not have any weakly quasicontinuous representative.
Hence \ref{b-outer}\negimp\ref{b-wqcont} and 
thus also \ref{b-outer}\negimp\ref{b-qcont}.
\end{example}

\begin{example} \label{ex-int-Loo}
We are now going to give an example with the same properties
as Example~\ref{ex-2^{-n}}, but with a nonatomic measure.

Consider the compact space 
\begin{equation*} 
  \P = \bigcup_{n=0}^\infty I_n,
  \quad \text{where} \quad
  I_n=\begin{cases}
     \{0\}, & \text{if } n=0, \\
     [2^{-1-2n}, 2^{-2n}], & \text{if } n\ge 1,
    \end{cases}
\end{equation*}
endowed with the Euclidean distance and $1$-dimensional Lebesgue measure,
and with $X=L^\infty$.
Since there are no curves connecting the origin $0$ to any other
point $x \in \P$, the constant zero function is an upper gradient of $\chi_{\{0\}}$. 
Thus, $\chi_{\{0\}}\in\NinftyP$ and $\Cinfty(\{0\}) = 0$.
Obviously, $\chi_{\{0\}}|_{\P \setminus \{0\}}$ 
is continuous and hence $\chi_{\{0\}}$ is weakly quasicontinuous.

Each function $ u \in \NinftyP$ is Lipschitz continuous within each interval $I_n$, 
and thus continuous on $\P$, except possibly at $0$.
Hence  $u$ is weakly quasicontinuous, i.e.\ \ref{b-wqcont} holds.
It also follows, as in Example~\ref{ex-union-new}, 
that
\[
     \Cinfty(E)=\begin{cases}
        0, & \text{if $E=\emptyset$ or $E=\{0\}$}, \\ 
        1, & \text{otherwise}.
        \end{cases}
\]
Hence, $\Cinfty$ is
not an outer capacity for sets of zero capacity, i.e.\ \ref{b-outer-zero} fails. 

Moreover, every quasicontinuous function is continuous, by
Proposition~\ref{prop-Linfty-gen}.
On the other hand, if $E=\bigcup_{k=1}^\infty I_{2k}$,
then $\chi_{E} \in \NX$ does not have a continuous representative, and thus
not a quasicontinuous representative, i.e.\ \ref{b-repr} fails.
In particular, \ref{b-wqcont}\negimp\ref{b-outer-zero}
and \ref{b-wqcont}\negimp\ref{b-repr}.
The Vitali--Carath\'eodory property (VC) 
again fails for $\chi_{\{0\}}$.
\end{example}

\begin{example}  \label{ex-top-sine}
(Damped topologist's sine curve and $N^{1,\infty}$)
This is yet another example
with similar properties to Examples~\ref{ex-2^{-n}} and~\ref{ex-int-Loo},
but with a nonatomic measure and a pathconnected space~$\P$.
Let 
\[
     \P=\biggl\{\biggl(t,t\sin\frac{1}{t}\biggr): 0< t \le 1\biggr\} \cup \{(0,0)\},
\]
equipped with the Euclidean metric.
Then $\P$ is a compact pathconnected subset of $\R^2$,
but it is not rectifiably connected.

Let $\mu(E)=\LL^1(\pi(E))$, where 
$\pi(t,y)=t$ is the 
orthogonal projection of $\P$ onto $[0,1]$ 
and $\LL^1$ is the Lebesgue measure on $[0,1]$.
Also let $X=L^\infty$.

Since $\P$ has infinite length,
there are no (nonconstant)
rectifiable curves in $\P$ through $0:=(0,0)$.
It thus follows that $0$ is an
upper gradient of $\chi_{\{0\}}$
and that $\|\chi_{\{0\}}\|_{\NinftyP}=0$.
Hence $\Cinfty(\{0\})=0$.

On the other hand, if $v \in \NinftyP$ then 
$v$ is locally Lipschitz on $\P \setm \{0\}$.
Thus all functions in $\NinftyP$ are weakly  quasicontinuous, 
i.e.\ \ref{b-wqcont} holds.
We also conclude that $\Cinfty(\{x\})=1$ if $x \ne 0$, and 
hence $\Cinfty(G)=1$ for every nonempty open set $G$.
Thus $\Cinfty$ is not an outer capacity 
for sets of zero capacity, i.e.\ \ref{b-outer-zero} fails.
In particular, \ref{b-wqcont}\negimp\ref{b-outer-zero}.

Next, define  
\[
        h(x)= \dist (w(x),\Z),
        \quad \text{where} \quad
     w(x)=\begin{cases}
        l_x, & \text{if }x \ne 0, \\
        0, & \text{if }x=0,
        \end{cases}
\]
and 
$l_x$ is the length of the geodesic (in $\P$) between
$x$ and $(1,\sin 1)$.
Then $1$ is a
minimal $L^\infty$-weak 
upper gradient of $h$ and
$h \in \NinftyP$.
If $\htilde =h $ q.e., then $\htilde=h$ on $\P \setm \{0\}$
and since
\begin{equation}   \label{eq-discont}
0 = \liminf_{x\to0} h(x) < \limsup_{x\to0} h(x) = \tfrac12,
\end{equation}
it follows that $\htilde$ is not 
continuous, and thus not quasicontinuous, 
by Proposition~\ref{prop-Linfty-gen}.
Hence
\ref{b-repr} fails and \ref{b-wqcont}\negimp\ref{b-repr}.

Moreover, if $f \in \NinftyP$ is  continuous, then
$\|f-h\|_{L^\infty(\P)} \ge \tfrac{1}{4}$ 
due to \eqref{eq-discont}.
Hence continuous functions are not dense in $\NinftyP$.
\end{example}

\begin{example} \label{ex-top-sine+int}
We modify Example~\ref{ex-top-sine} as follows.
Let 
$\P=\P^- \cup\P^+$, 
where
\[
  \P^-= [-1,0] \times \{0\}
 \quad \text{and} \quad
     \P^+=\biggl\{\biggl(t,t\sin\frac{1}{t}\biggr): 0< t \le 1\biggr\},
\]
equipped with the Euclidean metric.
Then $\P$ is a compact pathconnected subset of $\R^2$,
but it is not rectifiably connected.
We equip $\P$ with the measure $\mu$  such that
$\mu|_{\P^-}$ is the Lebesgue measure and 
$\mu|_{\P^+}$ is the  
measure in Example~\ref{ex-top-sine}.
Note that $\mu$ is nonatomic.
Let $X=L^\infty$.

Here $\Cinfty(E) =1$ whenever $E \ne \emptyset$,
and thus $\Cinfty$ is an outer capacity, i.e.\ \ref{b-outer} holds.
Moreover any weakly quasicontinuous function must be continuous.
However, $\chi_{\P^+},\chi_{\P^-} \in \NinftyP$  
and they are 
not even weakly quasicontinuous, i.e.\ 
\ref{b-wqcont} fails.
In particular, \ref{b-outer}\negimp\ref{b-wqcont} 
and thus also \ref{b-outer}\negimp\ref{b-qcont}.
\end{example}

\section{Non-\texorpdfstring{$L^\infty$}{Loo} examples}
\label{sect-ex-non-Loo}

\begin{example} \label{ex-union-new-b} 
(Weighted $L^\infty$)
Let $\P=\bigcup_{n=0}^\infty I_n$ be as in Example~\ref{ex-union-new}.
Recall that $\P$ is compact.
If $L^\infty(\P)$ is for some $c >1$ replaced by
\[
     \|u\|_X= \|uw\|_{L^\infty(\P)},
     \quad \text{where} \quad
    w=\begin{cases}
      1/c  & \text{in } I_0, \\
      1,    & \text{otherwise,}
      \end{cases}
\]
then
\[
     \CX(E)=\begin{cases}
       0, & \text{if } E =\emptyset, \\
       1/c, & \text{if } \emptyset \ne E \subset I_0, \\
       1, & \text{if } E \setm I_0 \ne \emptyset,
       \end{cases}
\]
showing that $\CX$ is a $c$-quasiouter capacity, with $c$ being optimal,
i.e.\ \ref{b-qouter} holds while \ref{b-outer} fails.
In particular, \ref{b-qouter}\negimp\ref{b-outer}.
Since the norms $\|\cdot\|_X$ and $\|\cdot\|_{L^\infty(\P)}$ are equivalent,
\ref{df:BFS.finmeasfinnorm} holds while (AC) and \ref{b-wqcont} fail,
as in Example~\ref{ex-union-new}.
The Vitali--Carath\'eodory property (VC)
again fails for $\chi_{\{0\}}$.
\end{example}

\begin{example} \label{ex-int-not-Loo}
Let $\P=\bigcup_{n=0}^\infty I_n$ be as in Example~\ref{ex-int-Loo}.
Recall that $\P$ is compact.

(a)
Let $1 \le p < \infty$ and let $X$ be given by the norm 
\[
  \|u\|^p_{X} = \int_{\P} |u(x)|^p\,\frac{dx}{x}.
\]
Here, $\chi_{\P} \notin X$ and thus \ref{df:BFS.finmeasfinnorm}  fails, 
so $X$ is not a Banach function space, only 
a Banach function lattice.
Moreover, $\chi_{\{0\}}\in X$ shows that 
the Vitali--Carath\'eodory property (VC) fails.

As in Example~\ref{ex-int-Loo}, it follows that $\chi_{\{0\}} \in X$ and $\CX(\{0\})=0$.
Moreover, $\CX(G) = \infty$ for every
open neighbourhood $G$ of $0$ as $\chi_G \notin X$. 
Consequently, 
$\CX$ is not an outer capacity 
for sets of zero capacity, i.e.\ \ref{b-outer-zero} fails.

Let $u \in \NX$ with an upper gradient $g\in X$.
Since for every $[y,z]\subset I_n$ with $n=1,2,\dots$\,,
\begin{equation}    \label{eq-cont-on-In}
|u(y)-u(z)| \le \int_y^z g\,ds 
  = \int_y^z \frac{g(x)}{x^{1/p}} x^{1/p} \,dx
\le 2^{-2n} \biggl( \int_y^z g(x)^p \,\frac{dx}{x}\biggr)^{1/p}
\end{equation}
and the last integral tends to zero as $y \to z$,
we conclude that
$u$
is  continuous within each interval $I_n$, 
and thus continuous on $\P$, except possibly at $0$.

Next we shall show that
\begin{equation} \label{eq-ex-int-not-Loo}
\lim_{x \to 0+}  u(x)=0.
\end{equation}
Note first that
$\lim_{n \to \infty}  \inf_{I_n} |u|=0$.
Assume that $\limsup_{x \to 0+} |u(x)|>a>0$. 
Then there are infinitely many $n$ such that $\osc_{I_n}  |u| >a$.
For such $n$, by H\"older's inequality and since $g$ is an upper gradient, we
get as in 
\eqref{eq-cont-on-In} that
\[  
   a < \int_{I_n} g \, dx
     \le 2^{-2n} \biggl( \int_{I_n} g^p \, \frac{dx}{x} \biggr)^{1/p} 
     \le   2^{-2n} \|g\|_X, 
\] 
from which it follows that $g \notin X$, a contradiction.
Hence
we have shown \eqref{eq-ex-int-not-Loo}.
It then follows 
that $\ut=u\chi_{\P \setm \{0\}}$ is continuous on $\P$ and
$\ut=u$ q.e., i.e.\ \ref{b-repr} holds and \ref{b-repr}\negimp\ref{b-outer-zero}.
Moreover, since 
\[
   \lim_{k \to \infty} \|u-u_k\|_{\NX} = 0, 
\quad \text{where } 
u_k(x) = \begin{cases} 
    u(x), &  \text{if } x \in \bigcup_{n=1}^k I_n, \\
    0, & \text{otherwise.}
   \end{cases}
\]
and the functions $u_{k}$ 
are continuous, we see that
continuous functions are dense in $\NX$.

(b)
If we consider $X=L^p(\P,\mu)$, where $d\mu=dx/x$, then 
we obtain the exact same space as in (a), but this time as a standard 
 $L^p$ space.
Note however that this measure fails the ``standing assumption'' that $\mu(B)<\infty$
for every ball $B$.
This shows that the assumption that balls have finite measure  is essential for at least
some of the Newtonian $\Np$ theory in the literature.
For this interpretation, \ref{df:BFS.finmeasfinnorm}  holds
and (VC) fails.

(c)
Yet another (non-$L^\infty$) modification of this example is
letting 
$X$ be given by the (maximal function) norm
\[
  \|u\|_X =  \sup_{r>0} \frac{1}{\LL^1(\P \cap (0, r))}  
   \int_{\P \cap (0, r)} |u(x)|\,dx,
\]
where $\LL^1$ is 
the Lebesgue measure.
Again, $\chi_{\{0\}}\in\NX$ and $\CX(\{0\}) = 0$.

Now, let $G \subset \P$ be an open set containing $0$.
Then $\CX(G) \ge \|\chi_G \|_X \ge 1$,
and thus $\CX$ is not an outer capacity for sets of zero capacity,
i.e.\ \ref{b-outer-zero} fails.

As $\|u\|_X \ge \|u\|_{L^1(\P)}$ it follows essentially from
the first inequality in~\eqref{eq-cont-on-In}
that each 
$u \in \NX$ is  continuous within each interval $I_n$, 
and thus continuous on $\P$, except possibly at $0$.
Thus $u|_{\P \setm \{0\}}$ is continuous, and so $u$ is weakly quasicontinuous,
i.e.\ \ref{b-wqcont} holds and 
\ref{b-wqcont}\negimp\ref{b-outer-zero}.

Finally, let
$v=\chi_E$, where $E=\bigcup_{k=1}^\infty I_{2k}$.
Then $v \in \NX$ and $v$ does not have any quasicontinuous representative.
Hence \ref{b-repr} fails and 
\ref{b-wqcont}\negimp\ref{b-repr}.
The function $v$ also shows that continuous functions
are not dense in $\NX$.

(d) It is fairly easy to modify (a)--(c) using instead
the pathconnected set $\P$ in Example~\ref{ex-top-sine}.
We leave the details to the interested reader.
\end{example}

\begin{example} \label{ex-6.13}
For each $m=1,2,\dots$\,, let
\[
  \P^-_m = \biggl[ - \frac{1}{m}, 0 \biggr] \times \{ m \} \quad \text{and} \quad
	\P^+_m = \bigcup_{n=0}^\infty [2^{-1-2n}, 2^{-2n}] \times \{ m \}
\]
be a collection of line segments in $\R^2$. Then, we define
\[
  \P^- = \bigcup_{m=1}^\infty \P^-_m, \quad
  \P^+ = \bigcup_{m=1}^\infty \P^+_m  \quad \text{and} \quad
	\P = \P^- \cup \P^+.
\]
Let $\P$ be endowed with the Euclidean distance and 1-dimensional 
Lebesgue 
measure.
Note that $\P$ is proper.
Given a measurable function $u$, let 
\[
  \|u\|_X = \|u\|_{L^1(\P^-)} + \|u\|_{L^\infty(\P^+)}.
\]
Then, 
\[
  C_X(\P_m^-) = \frac{1}{m}, \quad
  C_X(\P_m^+) = 1
\quad\text{and}\quad
	C_X(\{x\}) =
	\begin{cases}
	  1/m, & \text{if } x \in \P_m^-, \\
		1,   & \text{if } x \in \P_m^+.
	\end{cases}
\]
For $\P_m^+$ and $x \in \P_m^+$, 
this is shown as in Example~\ref{ex-union-new}, 
while for $\P_m^-$ and $x \in \P_m^-$, 
the upper bound is similar and the lower bound
 follows from the estimate
\[
\|u\|_{N^{1,1}(\P_m^-)} \ge \int_{\P_m^-} g\,ds + \frac{1}{m} \inf_{\P_m^-}|u|
\ge 1- \inf_{\P_m^-} |u| + \frac{1}{m} \inf_{\P_m^-} |u| \ge \frac{1}{m}
\]
for every function $u$ with an upper gradient $g$, admissible in defining 
$C_X(\P_m^-)$ or $C_X(\{x\})$.

In particular, every weakly quasicontinuous function $u$ is continuous
(and even Lipschitz continuous on each $\P_m^+$), 
i.e.\ \ref{b-weak-qcont} holds.

For any $m=1,2,\dots$\,, if $G \supset \P_m^-$ is open, then
$G \cap \P_m^+ \neq \emptyset$. Therefore,
\[
   \CX(\P_m^-) = \frac{1}{m}  < 1 \le \CX(G).
\]
Letting $m \to \infty$ shows that $\CX$ is not a quasiouter capacity,
i.e.\ \ref{b-qouter} fails and 
\ref{b-weak-qcont}\negimp\ref{b-qouter}.
The functions $\chi_{\P_m^-} \in \NX$, $m \ge 2$, show that 
\ref{b-wqcont} fails,
that the Vitali--Carath\'eodory property (VC) fails,
and that continuous functions are not dense in $\NX$.
\end{example}

\section{The variable exponent
case \texorpdfstring{$X=L^{p(\cdot)}$}{X=Lp(.)}}

We shall now modify Example~\ref{ex-top-sine} to provide non-quasicontinuous
Newtonian functions based on the variable exponent space $L^{p(\cdot)}$.
Several of the other examples can also be modified similarly, but we
restrict ourselves to this example to illustrate the main ideas.

Recall that the Luxemburg norm on $L^{p(\cdot)}(E)$,
where $E$ and $p: E \to [1, \infty]$ are measurable,
 is given by
\[
\|f\|_{L^{p(\cdot)}(E)} = \inf \biggl\{ \la>0: 
   \int_{E} \biggl(\frac{|f(x)|}{\la}  \biggr)^{p(x)} \,d\mu(x) \le 1 \biggr\}
\]
and that $L^{p(\cdot)}(E)= \{f: \|f\|_{L^{p(\cdot)}(E)}<\infty\}$.
When $p(x)=\infty$, we use the convention that 
\[
   a^\infty=\begin{cases}
    0, & \text{if } 0 \le a<1, \\
    1, & \text{if } a=1, \\
    \infty, & \text{if } a>1.
   \end{cases}
\]

Our examples will use the following characterization.

\begin{prop}
\label{pro:Lp(x)_ACnorm}
Assume that 
$p: \P \to [1, \infty]$ is measurable. 
Let $E \subset \P$ be a measurable set with $\mu(E) < \infty$ and 
let 
\[
E_n = \{ x \in E: p(x) > n \}.
\]
Then, 
\[
  e^{p(\cdot)} \in \bigcap_{q\in[1,\infty)} L^q(E)
	\quad\text{if and only if} \quad
	\| \chi_{E_n} \|_{L^{p(\cdot)}(E)} \to 0\ \ \text{as }n\to \infty.
\]
Moreover, if $\|e^{p(\cdot)}\|_{L^{q_0}(E)} = \infty$ for some $q_0 \in [1, \infty)$,
then $\| \chi_{E_n} \|_{L^{p(\cdot)}(E)} \ge e^{-q_0}$ for every $n=1,2,\dots$\,.
\end{prop}

\begin{proof}
Suppose that $e^{p(\cdot)} \in L^q(E)$ for every $q \in [1, \infty)$. Then,
$p(\cdot) < \infty$ a.e.\@ in~$E$, which implies that $\mu(E_n) \to 0$ as
$n\to \infty$, since $\mu(E)<\infty$. 
Given an arbitrary $\lambda \in (0, 1/e)$, let 
$q_\lambda =|{\log \lambda}| \in (1, \infty)$. 
By the Cauchy--Schwarz inequality,
\begin{equation*} 
\int_{E} \biggl(\frac{\chi_{E_n}}{\lambda}\biggr)^{p(x)}\,d\mu(x) 
	= \int_E \bigl(e^{p(x)}\bigr)^{q_\lambda} 
\chi_{E_n} \,d\mu(x) 
\le \|e^{p(\cdot)}\|_{L^{2 q_\lambda}(E)}^{q_\lambda} \mu(E_n)^{1/2}.
\end{equation*} 
In particular, if $\mu(E_n) \le 1/\|e^{p(\cdot)}\|_{L^{2 q_\lambda}(E)}^{2q_\lambda}$,
then this implies that $\| \chi_{E_n} \|_{L^{p(\cdot)}(E)} \le \lambda$. 
Hence,
\[
   0 \le \liminf_{n\to\infty} \| \chi_{E_n} \|_{L^{p(\cdot)}(E)} 
   \le \limsup_{n\to\infty} \| \chi_{E_n} \|_{L^{p(\cdot)}(E)} \le \lambda.
\]
Passing to the limit $\lambda \to 0$ proves that $\| \chi_{E_n} \|_{L^{p(\cdot)}(E)}
\to 0$ as $n\to\infty$.

For the converse, we need to distinguish two cases based on the measure of
the set
$E_\infty \coloneq \{x\in E: p(x) = \infty \}$. If $\mu(E_\infty) > 0$, then
\[
  \|\chi_{E_n}\|_{L^{p(\cdot)}(E)} \ge 
\|\chi_{E_\infty}\|_{L^{p(\cdot)}(E_\infty)}
        = \|\chi_{E_\infty}\|_{L^{\infty}(E_\infty)} 
= 1.
\]
Suppose now that $\mu(E_\infty)=0$ and hence $\mu(E_n)\to 0$ as $n\to\infty$,
since $\mu(E)<\infty$.
Assume that $q_0 \in (1, \infty)$ is such that $e^{p(\cdot)} \notin L^{q_0}(E)$.
Then for all $n=1,2,\dots$\,,
\begin{equation*}
\int_{E} \biggl(\frac{\chi_{E_n}}{e^{-q_0}}\biggr)^{p(x)}\,d\mu(x) 
\ge \int_{E} e^{q_0 p(x)} \,d\mu(x) - e^{q_0 n} \mu(E\setm E_n)
= \infty.
\end{equation*}
Therefore, $\|\chi_{E_n}\|_{L^{p(\cdot)}(E)} \ge e^{-q_0}$ and 
passing to the limit yields immediately that
\begin{equation*}  
  \lim_{n\to\infty} \|\chi_{E_n}\|_{L^{p(\cdot)}(E)} \ge e^{-q_0} > 0. \qedhere
\end{equation*}
\end{proof}

\begin{example}
\label{exa:nonqcont-N1p(x)}
\textup{(Damped topologist's sine curve and $N^{1,p(\cdot)}$)}
Let $\P$ and $\mu$ be as in Example~\ref{ex-top-sine}. 
Let $p: \P \to [1, \infty]$
be a measurable variable exponent function that is finite a.e. 
Suppose that there is
$q_0 \in (1, \infty)$ such that 
$\| e^{p(\cdot)} \|_{L^{q_0}(\P)} = \infty$, but 
\[
\| e^{p(\cdot)} \|_{L^q(\P \setminus G)} < \infty
\quad \text{for every } q \ge q_0
\text{ and every open neighbourhood $G$ of $0$}.
\]

For instance, one may consider $p(x)=f(t)$, with $x=(t,\sin(1/t))$
and  $1 \le p((0,0)) \le \infty$, 
where $f\:(0,1]\to[1,\infty)$ is locally bounded and
there is $t_0>0$ such that
\[
f(t)> \frac{|{\log t}|}{q_0} 
\quad \text{for all $t \le t_0$}.
\]
Indeed,
for $q\ge q_0$,
\begin{equation*}
\int_{\P} e^{p(x)q} \,d\mu(x) 
   = \int_0^\infty \LL^1(\{t : f(t)  > q^{-1} \log z\}) \,dz 
\ge \int_{1/t_0}^\infty \frac{dz}{z} =\infty.
\end{equation*}
Moreover, for any open neighbourhood $G$ of $0$,
we have
\[
\int_{\P\setm G} e^{p(x)q} \,d\mu(x) \le e^{p_G q} <\infty,
\quad \text{where } p_G:=\sup_{\P \setm G} p<\infty.
\]

Since there is no (nonconstant) rectifiable curve in $\P$ going through $0$,
the function $g(x) = 0$ is an upper gradient of $u = \chi_{\{0\}}$. Therefore,
$u \in N^{1, p(\cdot)}(\P)$ with $\|u\|_{N^{1,p(\cdot)}(\P)} = 0$. It also
follows that $C_{p(\cdot)}(\{0\}) = 0$.

Let $G \subset \P$ be an open set containing 
$0$ and let $v \in N^{1, p(\cdot)}(\P)$
satisfy $v \ge 1$ on $G$. Then,
\begin{equation} \label{eq-Np(x)}
  \| v \|_{N^{1,p(\cdot)}(\P)} \ge \| v \|_{L^{p(\cdot)}(\P)}
	\ge \| \chi_G \|_{L^{p(\cdot)}(\P)} \ge e^{-q_0},
\end{equation}
by the assumption $\| e^{p(\cdot)} \|_{L^{q_0}(\P)} = \infty$
and Proposition~\ref{pro:Lp(x)_ACnorm}. 
Thus, $C_{p(\cdot)}(G) \ge e^{-q_0}$.
In particular,
\[
  C_{p(\cdot)}(\{0\}) = 0 < e^{-q_0} \le \inf_{\substack{ G \supset \{0\} \\
		G \text{ open}}} C_{p(\cdot)}(G),
\]
which implies that $C_{p(\cdot)}$ is not an outer capacity, not even for
zero-capacity sets, i.e.\ \ref{b-outer-zero} fails.
It also follows from \eqref{eq-Np(x)} that the norm $\|\cdot\|_{L^{p(\cdot)}}$
is neither absolutely continuous (AC) nor satisfies 
the Vitali--Carath\'eodory property (VC).

Arguing similarly as in Example~\ref{ex-int-not-Loo}(a)
shows that $u|_{\P \setm \{0\}}$ is continuous
for every $u \in N^{1,p(\cdot)}(\P)$.
Hence all functions in $N^{1,p(\cdot)}(\P)$ are weakly quasicontinuous,
i.e.\ \ref{b-wqcont} holds.
On the other hand, the function $h \in N^{1,p(\cdot)}(\P)$ from
Example~\ref{ex-top-sine}
has no quasicontinuous representative here either, i.e.\ \ref{b-repr} fails.
\end{example}

\section{The Sierpi\'nski carpet}
\label{sect-Sierpinski}

The damped topologist's sine curve and some of the disconnected 
examples were based on the fact that there are no nonconstant rectifiable 
curves passing through the origin.
The following example is well connected, even 2-quasiconvex, and yet
\ref{b-qcont}--\ref{b-wqcont}
in Theorem~\ref{thm-qcont-char} fail.
In particular, there is a function in $\Ninfty(\P)$ which is not even weakly
quasicontinuous.

\begin{example} \label{ex-Sierpinski-new}
Let $\P\subset [0,1]^2$ be the Sierpi\'nski carpet 
$S_{1/3}$,
equipped with the
$\tfrac{\log8}{\log3}$-dimensional 
Hausdorff measure $\mu$,
normalized so that $\mu(S_{1/3})=1$.
Recall that the Sierpi\'nski carpet 
$S_{1/3}$ 
is a selfsimilar fractal constructed by 
dividing the square $[0,1]^2$ into nine squares with side length $\tfrac13$ and 
removing the open middle square, 
and then repeating this process in each of the remaining eight squares,
in a similar way as when constructing the 
Cantor ternary set. 

We are going to show that, for every quasi-Banach function lattice $X$,
there are $X$-almost no nonconstant rectifiable curves in $\P$,
and therefore  $\NX(\P)=X(\P)$.
This is well known for $X=L^p$, but may not have been
observed for general quasi-Banach function lattices.

Let $\pi_1(x_1,x_2)=x_1$ be the orthogonal projection of $\P$
 onto the interval $[0,1]$.  
It follows from the proof of Proposition~4.5 in Bourdon--Pajot~\cite{BouPaj}
(see also Corollary~4.15 in 
Durand-Cartagena--Jaramillo--Shan\-mu\-ga\-lin\-gam~\cite{DurJarShan})
that the Lebesgue measure $\LL^1$ on $[0,1]$ is 
singular with respect to the pushforward $\mus$ of $\mu$,
defined by
\[
\mus(A) = \mu(\pi_1^{-1}(A)) \quad \text{for } A\subset [0,1].
\]

Indeed,  by the Radon--Nikodym theorem,
\[
\mus = h\LL^1 +\sigma, \quad \text{where }
 h\in L^1(0,1) 
\text{ and } 
\sigma \perp \LL^1.
\]
Now, if $(a_1a_2\dots)_3$ with $a_j\in\{0,1,2\}$ is a ternary expansion
of $x\in[0,1]$ and $N_n(x)$ is the number of $1$'s among $a_1,\dots, a_n$,
$n=1,2,\dots$\,, then it is not difficult to verify using the construction
of $S_{1/3}$ 
that for a.e.\ $x\in [0,1]$,
\[
h(x) = \lim_{n\to\infty}   \frac{\displaystyle \bigl( \tfrac14 \bigr)^{N_n(x)}
\bigl( \tfrac38 \bigr)^{n-N_n(x)}} {\displaystyle \bigl( \tfrac13 \bigr)^{n}}
= \lim_{n\to\infty}   \Bigl( \frac23 \Bigr)^{N_n(x)}
\Bigl( \frac98 \Bigr)^{n}.
\]
Thus, for a.e.\ $x\in [0,1]$,
\[
h(x) = 0
    \quad   \text{if } \displaystyle \limsup_{n\to\infty} \frac{N_n(x)}{n} > 
\theta_0 := \frac{\log \tfrac98}{\log \tfrac32} =0.29\dots < \frac13.
\]
Now,   $N_n(x)$ can be written as the sum
\[
\sum_{j=1}^n {\bf 1}_{\{a_j=1\}}
\]
of independent identically distributed random variables on the probability space
$[0,1]$ with the Lebesgue measure and the Borel $\sigma$-algebra.
Since the expectations
\[
{\bf E} ({\bf 1}_{\{a_j=1\}}) = \tfrac13,
\]
the strong law of large numbers immediately gives that
\[
\LL^1 \biggl( \biggl\{ x\in[0,1]: \lim_{n\to\infty} \frac{N_n(x)}{n} = \frac13 \biggr\} \biggr) =1.
\]
It therefore follows that $h=0$ a.e.\ in $[0,1]$.
Hence, $\mus = \sigma$ is
singular with respect to the Lebesgue measure $\LL^1$.
Moreover, $\supp \mus = [0,1]$ and $\mus$ is nonatomic.
Thus, for each (nondegenerate) interval $[a,b] \subset [0,1]$,
there is a Borel set $N_{a,b}\subset [a,b]$
such that
$\LL^1(N_{a,b})>0$ and $\mus(N_{a,b})=0$.

Next, let $\Ga$ be the collection
of all nonconstant  rectifiable curves $\ga:[0,l_\ga] \to \P$,
and  let
\[
\Gav= \{\ga \in \Ga : \pi_1(\ga) \text{ is a singleton}\}
\]
be the collection of the ``vertical'' curves.
Let $\pi_2(x_1,x_2)=x_2$ be the orthogonal projection on the $y$-axis,
\[
    N=\bigcup_{\substack{a,b\in \Q \\ 0 \le a <b\le 1}} N_{a,b}
\quad \text{and} \quad
E=\pi_1^{-1}(N) \cup \pi_2^{-1}(N).
\]
Then 
$\rho:=\infty \chi_E=0$ $\mu$-a.e.\ and thus
$\rho\in X(\P,\mu)$ for every quasi-Banach function lattice $X$, by (P1).
Moreover, 
\[
    \int_\ga \rho \, ds \ge \int_{\pi_1(\ga)} \infty \chi_N(t) \, dt = \infty
   \quad \text{if }\ga \in \Ga \setm \Gav,
\]
while for $\ga \in \Gav$   
the same estimate holds with $\pi_2(\ga)$ 
instead of $\pi_1(\ga)$.

Hence there are $X$-almost no nonconstant rectifiable curves in $\P$.
Therefore 
\[
\NX(\P)=X(\P)
\quad \text{and} \quad 
\|\cdot\|_{\NX(\P)} = \|\cdot\|_{X(\P)}.
\]
In particular, for $X=L^\infty$, it follows from 
the last part of
Theorem~\ref{thm-von-Koch} 
that \ref{b-outer-zero} fails.

Moreover, since $\mu(\P\setm E)=\mu(\P)>0$ 
and $\mu$ is Borel regular, there is a compact
set $K \subset \P \setm E$ with $\mu(K)>0$, see 
e.g.\ Rudin~\cite[Theorem~2.18]{RudinRC}.
As $E$ is dense in $\P$, the compact set $K$ has empty interior and hence
$v=\chi_K$ is not weakly quasicontinuous, by Theorem~\ref{thm-von-Koch}\ref{q2}.
\end{example}

\end{document}